\documentclass[20pt, oneside]{article}   	
\usepackage{geometry}                		
\usepackage{amsmath}
\usepackage{amsthm}	
\usepackage{amssymb}
\usepackage{tikz}
\usepackage[implicit]{hyperref}
\usepackage[bottom]{footmisc}
\usepackage[affil-it]{authblk}
\usepackage[toc,page]{appendix}
\newtheorem{thm}{Theorem}[section]

\newtheorem{prop}[thm]{Proposition}
\newtheorem{note}[thm]{Note}
\newtheorem{lem}[thm]{Lemma}
\newtheorem{cor}[thm]{Corollary}
\newtheorem*{claim*}{Claim}

\newtheorem*{thm*}{Theorem}

\newtheorem{term}[thm]{Terminology}

\newtheorem{ex}[thm]{Example}
\newtheorem{man}{Theorem}

\theoremstyle{definition}
\newtheorem*{rem*}{Remark}
\newtheorem{defi}[thm]{Definition}
\newtheorem{rem}[thm]{Remark}
\newtheorem{conv}[thm]{Convention}

\newcommand{\F}{\mathbb{F}}

\newcommand{\fq}{\mathfrak{q}}

\newcommand{\fp}{\mathfrak{p}}

\newcommand{\M}{\mathcal{M}}
\newcommand{\rsep}{r_{\rm sep}}

\newcommand{\hookdownarrow}{\mathrel{\rotatebox[origin=c]{-90}{$\hookrightarrow$}}}
\newcommand{\xdownarrow}[1]{%
  {\left\downarrow\vbox to #1{}\right.\kern-\nulldelimiterspace}
}

\title{On singular moduli for higher rank Drinfeld modules}
\author{Chien-Hua Chen \thanks{Electronic address: \texttt{danny30814@ncts.ntu.edu.tw}; ORCID: \texttt{0000-0003-3267-5603} ; Corresponding author}}
\affil{Mathematics Division,\\ National Center for Theoretical Sciences,\\ Taipei, Taiwan}

\begin{document}

\maketitle
\abstract
In this paper, we generalize Dorman's work to estimate singular moduli for higher rank Drinfeld modules. In particular, we give a lower bound on the valuation of singular moduli for Drinfeld modules with complex multiplication by an imaginary field extension over the rational function field. Furthermore, we compute several examples for rank-$3$ case.

\section {Introduction}

In the theory of arithmetic of elliptic curves, a singular modulus is defined to be the $j$-invariant of an elliptic curve with complex multiplication. In \cite{GZ85}, Gross and Zagier introduced a systematic method to estimate the difference of two singular moduli: 

\begin{thm*}[\cite{GZ85}, Theorem 1.3]
Let $d_1$, $d_2$ be two fundamental discriminant of imaginary quadratic fields. Assume further that $d_1$ and $d_2$ are coprime to each other. Let $w_1$, $w_2$ be the number of roots of unity in the quadratic orders of discriminant $d_1$, $d_2$, respectively. Define
$$J(d_1,d_2):=\left( \prod_{[\tau_1], [\tau_2]\ {\rm with\ disc }(\tau_i)=d_i} (j(\tau_1)-j(\tau_2))      \right)^{\frac{4}{w_1w_2}}.$$

Then we have
$$J(d_1,d_2)^2=\pm\prod_{x\in\mathbb{Z}\ \textrm{and } n,n'\in\mathbb{Z}_{>0} \textrm{ with }x^2+4nn'=d_1d_2}n^{\epsilon(n')}.$$
Here $\epsilon(n')$ is an explicit map defined in terms of  Legendre symbols.

\end{thm*}

We refer a more detailed explanation of the proofs to \cite{C04}. Subsequently, the Gross-Zagier estimation aligns with intriguing topics, including the Hilbert class polynomial for CM elliptic curves \cite{BJO06}, the number of supersingular reduction primes of an elliptic curve over $\mathbb{Q}$ \cite{E87}, and N\'eron-Tate local height on Heegner points \cite{GZ86}. Additionally, one can also viewed Gross-Zagier singular moduli formula geometrically as valuation of CM cycle on modular stack over $\mathbb{Z}$ of pairs of elliptic curves (see section 3 in \cite{T10}).

Shifting focus to the analogy of Drinfeld modules, we fix some notation first.  Let $r\geqslant 2$ be an integer. Let $A=\F_q[T]$ be the polynomial ring of variable $T$ over the finite field $\F_q$ with $q=p^e$ a prime power, $F=\F_q(T)$ be its fractional field, and fix a prime ideal $\fp=(\pi)$ of $A$ with monic generator $\pi$. Let $K/F$ be an imaginary extension of degree $r$, we denote $r_{\rm sep}=[K:F]_{\rm sep}$ to be its separable degree, and let $\mathcal{O}_K$ be the integral closure of $A$ in $K$.  We define ${\rm CM}(\mathcal{O}_K,\iota)$ to be the set of isomorphism classes of normalizable Drinfeld modules having CM by $\mathcal{O}_K$, where $\iota$ is a fixed embedding from $K$ to $\bar{F}$, the algebraic closure of $F$, see Definition \ref{cmdef}. And let $J^{(\delta_1,\cdots,\delta_{r-1})}$ be a basic $J$-invariant of rank-$r$ Drinfeld modules, see Definition \ref{basicj}.

Here we always assume $K/F$ is a {\bf normal} extension. As a result, we can estimate the valuation of difference of basic $J$-invariants between (i) those rank-$r$ Drinfeld modules $\phi$ with CM by $\mathcal{O}_K$ and (ii) a fixed Drinfeld module $\phi'$ of rank $r$ defined over a finite normal extension over $\F_q(T)$:

\begin{man}[Theorem \ref{main1'}]\label{intro1}

Fix a  Drinfeld module $\phi'$ of rank $r$ defined over a normal extension $K'/F$ that has good reduction at a place $\fq'$ of $K'$ lying above $\fp$. We have
$$
{\rm ord}_\fp\left(\prod_{[\phi]\in{\rm CM}(\mathcal{O}_K,\iota)}J^{(\delta_1,\cdots,\delta_{r-1})}(\phi)-J^{(\delta_1,\cdots,\delta_{r-1})}(\phi')\right)
\geqslant \frac{q-1}{\rsep\cdot(q^r-1)\cdot e_{K\cdot K',\fp}}\sum_{n\geqslant 1} \#\mathcal{M}_n 
.
$$
Here $e_{K\cdot K',\fp}$ is the ramification index of $\fp$ in $K\cdot K'$, and $ \mathcal{M}_n$ is the set of $A$-algebra embeddings $\eta:\mathcal{O}_K\hookrightarrow{\rm End}_{W_n}(\phi')$ with $\eta|_A=\phi'$ over $W_n$, where $W_n$ is the local Artinian ring defined in Convention \ref{wn}.
\end{man}

When $\phi'$ is the particular  Drinfeld module $\varphi$ introduced in section 2.2, its basic $J$-invariants are always vanishing. We can observe that

$$\prod_{[\phi]\in {\rm CM}(\mathcal{O}_K,\iota)}\left(J^{(\delta_1,\cdots,\delta_{r-1})}(\phi)-J^{(\delta_1,\cdots,\delta_{r-1})}(\varphi)\right)=\prod_{[\phi]\in {\rm CM}(\mathcal{O}_K,\iota)}J^{(\delta_1,\cdots,\delta_{r-1})}(\phi).$$

In this case, we can make a finer estimation:
\begin{man}[Theorem \ref{main1}]\label{intro2}

\begin{align*}
{\rm{ord}}_\fp(J_{\mathcal{O}_K}^{(\delta_1,\cdots,\delta_{r-1})})&:={\rm{ord}}_\fp\left(\prod_{[\phi]\in {\rm CM}(\mathcal{O}_K,\iota)}J^{(\delta_1,\cdots,\delta_{r-1})}(\phi)\right)\nonumber\\
\ \\ 
&\geqslant \frac{(\sum_{i=1}^{r-1}\delta_i)(q-1)}{\rsep\cdot(q^r-1)\cdot e_{K, \fp}}\sum_{n\geqslant 1} \#\mathcal{M}_n. \\
\end{align*}

Here $e_{K,\fp}$ is the ramification index of $\fp$ in $K$, and $\mathcal{M}_n$ is the set of $A$-algebra embeddings $\eta:\mathcal{O}_K\hookrightarrow {\rm End}_{W_n}(\varphi)$ such that $\eta|_A=\varphi$ over $W_n$, where $W_n$ is the local Artinian ring defined Convention \ref{wn}.
\end{man}

\begin{rem*}
If rank $r=2$, there is only one $j$-invariant $j=J^{(q+1)}$. This $j$-invariant can distinguish rank-$2$ Drinfeld modules up to isomorphism. Our estimation on ${\rm ord}_\fp(J^{(q+1)}_{\mathcal{O}_K})$ can be reduced into the equality introduced by Dorman in \cite{D91}, equation (5.5). We refer to Remark \ref{rk2} for more details.

However, some technical issues appear when the rank of Drinfeld module is greater than $2$.  For instance, the structure of coarse moduli scheme for rank-$r$ Drinfeld module with $r>2$ is quite different from the rank-$2$ case. The spectrum of such a coarse moduli scheme is a ring generated by more than one invariant, see section 2.3 for a brief introduction. In other words, one single ``$j$-invariant'' can no longer distinguish isomorphism classes of Drinfeld modules of rank $r>2$. This observation naturally justifies why our estimation takes the form of an inequality. 
\end{rem*}

\begin{rem*}

By fixing a basic $J$-invariant $J^{(\delta_1,\cdots,\delta_{r-1})}$, one can define the so-called ``Hilbert class polynomial'' for Drinfeld modules as follows:
$$H_{\mathcal{O}_K}^{(\delta_1,\cdots,\delta_{r-1})}(X):=\prod_{[\phi]\in {\rm CM}(\mathcal{O}_K,\iota)}\left(X-J^{(\delta_1,\cdots,\delta_{r-1})}(\phi)\right)\in \mathcal{O}_K[X].$$
The $J_{\mathcal{O}_K}^{(\delta_1,\cdots,\delta_{r-1})}$ in our estimation is exactly the constant term of the Hilbert class polynomial up to a plus-minus sign. And its valuation at a finite prime $\fp$ is related to the study of supersingular reduction prime of a rank-$r$ Drinfeld module, which will be explored in the future.

\end{rem*}

Finally, when $\varphi \mod \fp$ is supersingular, the cardinality of $\mathcal{M}_n$ can be translated into a matrix counting problem. For simplicity, we may write $K=F(s_1,\cdots,s_t)$ and $\mathcal{O}_K=A[s_1,\cdots,s_t]$ with $t$ minimal. Thus we have an isomorphism $A[s_1,\cdots,s_t]\cong A[X_1\cdots,X_t]/I_{\mathcal{O}_K}$, where $I_{\mathcal{O}_K}$ is an ideal of the polynomial ring $A[X_1,\cdots,X_t]$.

In order to compute $\#\M_n$, we have (see Proposition \ref{punr} \& \ref{pram}) the following characterization of ${\rm End}_{W_n}(\varphi)$, where $W_n$ is the local Artinian ring introduced in Convention \ref{wn}:
$${\rm End}_{W_n}(\varphi)=\begin{cases}\F_{q^r}[T]+\pi^{n-1}{\rm End}_{\bar{\F}_\fp}({\varphi}), &\textrm{ if }\fp \textrm{ is unramified in }K/F\\
\ \\
\F_{q^r}[T]+\pi^{\lfloor\frac{n+e_{K,\fp}-1}{e_{K,\fp}}\rfloor-1}{\rm End}_{\bar{\F}_\fp}({\varphi}), &\textrm{ if }\fp \textrm{ ramifies in }K/F
\end{cases}$$
where $e_{K,\fp}$ is the ramification index of $\fp$ in $K$, and $\bar{\F}_{\fp}$ is the algebraic closure of $\F_\fp:=A/\fp$. In the following result, we give a matrix representation of the endomorphism ring ${\rm End}_{\bar{\F}_{\fp}}(\varphi)$, and describe $\#\M_n$ as a counting number in the computational aspect:

\begin{man}[Proposition \ref{endomu} \& Corollary \ref{estj}]\label{intro3}
Under an additional restriction that 
\begin{enumerate}

\item[$\bullet$] $\fp=(\pi)$ is a prime ideal of $\F_q[T]$ whose degree is coprime to $r$,

\end{enumerate}
We have
$$
\mathcal{M}:={\rm End}_{\bar{\F}_\fp}(\varphi)=\left\{   \left(\begin{array}{cccc}x_1 & x_2 & \cdots & x_r \\\pi\sigma(x_r) & \sigma(x_1) & \cdots & \sigma(x_{r-1}) \\\vdots &  & \ddots & \vdots \\ \pi\sigma^{r-1}(x_2) & \cdots & \pi\sigma^{r-1}(x_r) & \sigma^{r-1}(x_1)\end{array}\right)
         \    \Bigg|\    x_i\in \F_{q^r}[T]    \textrm{ for all }i \right\},           
$$
and

\begin{align*}
{\rm{ord}}_\fp(J_{\mathcal{O}_K}^{(\delta_1,\cdots,\delta_{r-1})})&\geqslant \frac{(\sum_{i=1}^{r-1}\delta_i)(q-1)}{r_{\rm sep}\cdot(q^r-1)\cdot e_{K, \fp}}\sum_{m\geqslant 0} \#\mathcal{M}_{m\cdot e_{K,\fp}+1} 
\end{align*}.

Here  $e_{K,\fp}$ is the ramification index of $\fp$ in $K$, and
$
\#\M_{m\cdot e_{K,\fp}+1}$ is equal to 

$\#\Bigg\{(\alpha_1,\cdots, \alpha_t)\in \mathcal{M}^t \bigg|\ $
$
\begin{array}{cc}
 \alpha_i(x_1,\cdots,x_r)\in\mathcal{M} \textrm{ satisfies } x_k\equiv 0 \mod \pi^{m} \textrm{ for } 2\leqslant k\leqslant r, \\ \alpha_i\alpha_j=\alpha_j\alpha_i \textrm{ for }1\leqslant i, j\leqslant t,  \\ f(\alpha_1,\cdots,\alpha_t)=0 \textrm{ for all } f\in I_{\mathcal{O}_K}  \\
 \end{array}
$
$\Bigg\}.$

\end{man}
\begin{rem*}
This result is particularly useful for computation when $\mathcal{O}_K$ is generated by only one element. More specifically, when $\mathcal{O}_K=A[s]$, the ideal  $I_{\mathcal{O}_K}$ becomes the principal ideal of $A[X]$ generated by the minimal polynomial of $s$. Hence the condition ``$f(\alpha)=0$ for all $f\in I_{\mathcal{O}_K}$'' can be reinterpreted using minimal polynomials, see Remark \ref{estone} for more details.  Based on the remark, we also provide computational examples in section 5.
\end{rem*}

Our paper is organized as follows: we start by the basic structures for Drinfeld modules in section 2. In section 3, we  compare the difference of  ``$j$-invariant'' of Drinfeld modules with number of isomorphisms between two Drinfeld modules, see Theorem \ref{thm1'}. Then we view rank-$r$ Drinfeld $\F_q[T]$-module with CM by $\mathcal{O}_K$ as rank-$1$ Drinfeld $\mathcal{O}_K$-module. Combining Drinfeld module analogue of Serre-Tate lifting theorem and Gross' lifting of height-$1$ formal $\mathcal{O}$-modules, we can construct a CM-lifting of a Drinfeld module over finite field. This transfer ``counting the number of isomorphisms between two Drinfeld modules'' into ``counting the number of endomorphisms in one explicit Drinfeld module with some additional conditions'', see Theorem \ref{thm2'}. Then we relate the counting number problem to count the number of $\F_q[T]$-algebra embeddings from $\mathcal{O}_K$ into an endomorphism ring, as stated in Proposition \ref{algemb}. 

In section 4, we add one more condition that $\fp=(\pi)$ is a prime ideal of $\F_q[T]$ whose degree is coprime to $r$, so that the Drinfeld module $\varphi \mod \fp$ is supersingular. Then we study the matrix representation of the endomorphism algebra associated to a supersingular Drinfeld $\F_q[T]$-module over a finite field. This endomorphism algebra takes the form of a central division algebra over $\F_q(T)$ which has dimension $r^2$. The counting problem on endomorphism on $\varphi \mod \mu^n$ can be established explicitly through this matrix representation. 

In Section 5, we compute several examples, including separable and inseparable cases, when rank $r=3$. In particular, in section 5.1 we show that the equality in Theorem 2 can be reached by selecting a specific basic $J$-invariant. However, we also demonstrate that a change in the choice of basic $J$-invariant results in a strict inequality for the bound.

\section{Preliminaries}

Let $A=\F_q[T]$ be the polynomial ring over finite field with $q=p^e$ an odd prime power, $F=\F_q(T)$ be the fractional field of $A$, and  $K$ be a finite extension over $F$. Set $K\{\tau\}$ to be the twisted polynomial ring with usual addition rule, and the multiplication rule is defined to be $\tau\alpha=\alpha^q\tau \text{ for any } \alpha\in K$.  

\subsection{Drinfeld modules}
We view $K$ as an $A$-field, which is a field equipped with a homomorphism $\gamma: A\rightarrow K$. The {\bf{$A$-characteristic}} of $K$ is defined to be the kernel of $\gamma$.

\begin{defi}
A Drinfeld $A$-module of rank $r$ over $K$ is a ring homomorphism

$$\phi: A\rightarrow K\{\tau\}$$
such that
\begin{enumerate}
\item[(i)] $\phi(a):=\phi_a$ satisfies ${\rm deg}_{\tau}\phi_a=r\cdot {\rm deg}_Ta$
\item[(ii)] Denote $\partial: K\{\tau\}\rightarrow K$ by $\partial(\sum a_i\tau^i)=a_0$, then $\phi$ satisfies $\gamma=\partial\circ\phi$.
\end{enumerate}
\end{defi}

From the definition of Drinfeld $A$-module, we can characterize a Drinfeld module $\phi$ by writing down $$\phi_T=T+g_1\tau+\cdots+g_{r-1}\tau^{r-1}+\Delta\tau^r, \textrm{ where } g_i\in K \textrm{ and }\Delta\in K^*.$$

\begin{defi}
Let $\phi$ and $\psi$ be two rank-$r$ Drinfeld $A$-modules over $K$. A {\bf{morphism}} $u:\phi\rightarrow \psi$ over $K$ is a twisted polynomial $u\in K\{\tau\}$ such that
$$u\phi_a=\psi_a u \text{ \rm for all } a\in A.$$ 
A non-zero morphism $u:\phi\rightarrow \psi$ is called an isogeny. A morphism $u:\phi\rightarrow \psi$ is called an {\bf{isomorphism}} if its inverse exists. Set ${\rm Hom}_K(\phi,\psi)$ to be the group of all morphisms $u:\phi\rightarrow \psi$ over $K$. Also, we denote ${\rm End}_K(\phi)={\rm Hom}_K(\phi,\phi)$. For any field extension $L/K$, we define
$${\rm Hom}_L(\phi,\psi)=\{u\in L\{\tau\} \mid u\phi_a=\psi_a u  \text{ \rm for all } a\in A \}.$$
For $L=\bar{K}$, we omit subscripts and write

$${\rm Hom}(\phi,\psi):={\rm Hom}_{\bar{K}}(\phi,\psi) \text{ \rm and } {\rm End}(\phi):={\rm End}_{\bar{K}}(\phi)$$

\end{defi}

\begin{rem}
The composition of morphisms makes ${\rm End}_L(\phi)$ into a subring of $L\{\tau\}$, called the {\bf{endomorphism ring}} of $\phi$ over $L$. Moreover, define an $A$-action on ${\rm Hom}_L(\phi,\psi)$ by $$a\circ u=u\phi_a=\psi_a u.$$ We can view the group of homomorphisms as an $A$-module. The $A$-action also implies $\phi(A)$ to be in the center of ${\rm End}_L(\phi)$. Thus we may view ${\rm End}_L(\phi)$ as an $A$-algebra via $\phi$.

\end{rem}

Now we can define the algebra induces from the endomorphism ring of $\phi$ as follows:

$${\rm End}^\circ (\phi):=F\otimes_A {\rm End}(\phi).$$

\begin{prop}

The algebra ${\rm End}^\circ (\phi)$ is a division algebra over $F$. If $K$ has  $A$-characteristic equal to zero, then ${\rm End}^\circ (\phi)$ is a field extension over $F$ which has a unique place over the place $\infty$ of $F$.  
\end{prop}
\begin{proof}
See Proposition 4.7.17 in \cite{G96}.
\end{proof}

\begin{defi}
We call a finite extension $L/F$ {\bf{imaginary}} if there is a unique place $\tilde{\infty}$ of $L$ over $\infty$.

\end{defi}
\begin{defi}
Let $K$ be a degree $r$ imaginary extension over $F$ with integral closure denoted by $\mathcal{O}_K$. Let $\mathcal{O}\subseteq \mathcal{O}_K$ be an $A$-order, we say that a rank-$r$ Drinfeld module $\phi$ over a field with $A$-characteristic equal to zero has {\bf{complex multiplication}} (CM) by $\mathcal{O}$ if

$${\rm End}(\phi)\cong \mathcal{O} \textrm{ as $A$-algebras}.$$

\end{defi}

We define the homomorphism space and endomorphism ring for reduction of Drinfeld modules. Let $\fq$ be a place of $K$. Consider the local field $K_\fq$ associated with normalized valuation $\nu$, discrete valuation ring $R$, and  normalized parameter $\pi$. Let $\phi:A\rightarrow K_\fq\{\tau\}$ be a Drinfeld module of rank $r$ over $K_\fq$. Up to isomorphism, we may assume $\phi$ is defined over $R$. Then $$\phi \mod \pi^n: A\rightarrow R/\pi^n R\{\tau\}$$ is a Drinfeld module over $R/\pi^nR$.

\begin{defi}
Let $\phi$ and $\psi$ be two rank-$r$ Drinfeld $A$-modules over $R/\pi^nR$.

\begin{enumerate}
 
\item[(i)] $${\rm Hom}_{R/\pi^nR}(\phi,\psi)=\{u\in R/\pi^nR\{\tau\} \mid u\phi_a=\psi_a u  \text{ \rm for all } a\in A \}.$$

\item[(ii)] $${\rm End}_{R/\pi^nR}(\phi)=\{u\in R/\pi^nR\{\tau\} \mid u\phi_a=\phi_a u  \text{ \rm for all } a\in A \}.$$
\end{enumerate}
\end{defi}

\begin{ex}

In this example, we review some properties on the rank-$r$ Drinfeld module $\varphi_T=T+\tau^r$ over $F$. This will be useful during our estimation of singular moduli.

\end{ex}
\begin{prop}\label{ssred}
\begin{enumerate}

\item[(i) ] The Drinfeld module $\varphi$ has CM by $\F_{q^r}[T]$, i.e. the  endomorphism ring ${\rm End}(\varphi)\cong \F_{q^r}[T]$.

\item[(ii) ]Consider the reduction $\underline{\varphi}$ of $\varphi$ modulo a prime $\fp$ of $A$.  The Drinfeld module $\underline{\varphi}$ is supersingular if and only if $\deg_T(\fp)$ is coprime to $r$. Here $\deg_T(\fp)$ is the $T$-degree of the monic generator of $\fp$.

\begin{proof}
\begin{enumerate}
\item[(i)] See Example 3.3.2 in \cite{P23}.

\item[(ii)]  We refer the necessary condition to Example 4.4.5 in \cite{P23}.

For the sufficient condition, suppose that $\deg_T(\fp)=d$ is coprime to $r$. From Lemma 3.2.11 in \cite{P23}, we have $${\rm{ht}}(\varphi_\fp)=H(\varphi)\cdot d, \textrm{ where $H(\varphi)$ is the height of $\varphi$.}$$
Because $\varphi_T=T+\tau^r$, the number ${\rm{ht}}(\varphi_\fp)$ is a multiple of $r$. Since $d$ is coprime to $r$, we must have $r\mid H(\varphi)$. This forces the height $H(\varphi)$ to be equal to $r$, which implies $\varphi$ has supersingular reduction at $\fp$.
\end{enumerate}
\end{proof}

\end{enumerate}
\end{prop}

\subsection{\texorpdfstring{$J$-}invariants for higher rank Drinfeld modules}\label{j}
We take the definition of $J$-invariants of higher rank Drinfeld modules introduced by Potemine \cite{P98}. Given a rank-$r$ Drinfeld module $\phi$ over an arbitrary $A$-field $K$, whose $A$-characteristic is equal to zero, with
$$\phi_T=T+g_1\tau+\cdots+g_{r-1}\tau^{r-1}+\Delta\tau^r, \textrm{ where }g_i\in K \textrm{ and }\Delta\in K^*.$$
We define a basic $J$-invariant of $\phi$ to be
$$J^{(\delta_1,\cdots,\delta_{r-1})}(\phi)=\frac{g_1^{\delta_1}\cdots g_{r-1}^{\delta_{r-1}}}{\Delta^{\delta_r}},$$
where $\delta_i$'s satisfy the following two conditions:
\begin{equation}
 \delta_1(q-1)+\delta_2(q^2-1)+\cdots+\delta_{r-1}(q^{r-1}-1)=\delta_r(q^r-1). \tag{1}
\end{equation}

\begin{equation}
0\leqslant \delta_i \leqslant \frac{q^r-1}{q^{{\rm{g.c.d.}}(i,r)}-1} \textrm{ for all }1\leqslant i\leqslant r-1;\ \ \text{ g.c.d.$(\delta_1,\cdots,\delta_r)=1.$} \tag{2}
\end{equation}

\begin{defi}\label{basicj}
We denote by $\left\{  J^{(\delta_1,\cdots, \delta_{r-1})}    \right\}$ the set of all basic $J$-invariants.

\end{defi}

\begin{thm}
$$M^r(1)={\rm Spec} A\left[\left\{  J^{(\delta_1,\cdots, \delta_{r-1})}    \right\}\right]$$ is the coarse moduli scheme of Drinfeld $A$-modules of rank $r$. Here $A\left[\left\{  J^{(\delta_1,\cdots, \delta_{r-1})}    \right\}\right]$ is the ring generated by  $\left\{  J^{(\delta_1,\cdots, \delta_{r-1})}    \right\}$ over $A$.
\end{thm}

\begin{proof}
See \cite{P98}, Theorem 3.1.

\end{proof}

\begin{rem}
Note that the set of generators $\left\{  J^{(\delta_1,\cdots, \delta_{r-1})}    \right\}$ of the  ring $A\left[\left\{  J^{(\delta_1,\cdots, \delta_{r-1})}    \right\}\right]$ is not algebraically independent. For instance consider $r=3$, and the basic $J$-invariants $J^{(q^2+q+1,0)}$, $J^{(0,q^2+q+1)}$, and $J^{(1,q)}$. We have
$$J^{(q^2+q+1,0)}\cdot (J^{(0,q^2+q+1)})^q=(J^{(1,q)})^{q^2+q+1}.$$ 
Note that the defining equation for $M^3(1)$ is explicitly written in section 7.1 of \cite{P98}.

\end{rem}

\section{Estimation on singular moduli}
From now on, we fix to the following terminologies:

\begin{enumerate}

\item[$\bullet$] Let $r\geqslant 2$ be an integer. And let $q=p^e$ be a prime power.

\item[$\bullet$] Let $F_\infty$ be the completion of $F$ at $\infty$, and $\mathbb{C}_\infty$ be the completion of algebraic closure of $F_\infty$

\item[$\bullet$] $K$ is a degree-$r$ imaginary and normal extension over $F$.

\item[$\bullet$]  Fix a generator $u\in \F_{q^r}^*$, and define $H=F(u)=\F_{q^r}(T)$. Its integral closure is denoted by $\mathcal{O}_H=\F_{q^r}[T]$.

\item[$\bullet$] Fix a prime ideal $\fp$ of $A$. Denote by $\nu_\fp$ a normalized valuation of $F_\fp$, the completion of $F$ at $\fp$. The valuation $\nu_\fp$ can be extended to any finite extension of $F_\fp$.

\end{enumerate}

\subsection{Complex multiplication and normalizable Drinfeld modules}

We first introduce the notion of normalizable Drinfeld modules.

\begin{defi}\label{cmdef}
\begin{enumerate}
\item[(i) ]

Let $\phi$ be a rank $r$ Drinfeld module over $\mathbb{C}_\infty$ with $CM$ by $\mathcal{O}_K$. We fix an isomorphism ${\rm End}(\phi)\cong \mathcal{O}_K$. The isomorphism composed with the derivative map $\partial$ induces a fixed embedding $\iota:\mathcal{O}_K\hookrightarrow \mathbb{C}_\infty$. From $\iota$ we extend uniquely to an embedding $\iota: K\hookrightarrow \mathbb{C}_\infty$. 
We say $\phi$ is normalizable if there is an isomorphism ${\rm End}(\phi)\cong \mathcal{O}_K$ such that $$\partial(\phi_\alpha)=\iota(\alpha) \textrm{ for all } \alpha\in\mathcal{O}_K.$$

\item[(ii) ]
Let ${\rm CM}(\mathcal{O}_K,\iota)$ be the set of isomorphism classes of normalizable Drinfeld modules over $\mathbb{C}_\infty$ having CM by $\mathcal{O}_K$.

\item[(iii)]
Let $H_K$ be the Hilbert class field over $K$, which is the maximal unramified extension over $\iota(K)\hookrightarrow \mathbb{C}_\infty$ where the place at infinity $\tilde{\infty}$ splits completely.

\end{enumerate}
\end{defi}

\begin{rem}
The cardinality of ${\rm CM}(\mathcal{O}_K,\iota)$ is equal to the class number $h_K$ of $K$. 

\end{rem}
Similar to the theory of elliptic curves, we have the following result on the defining field of Drinfeld modules with complex multiplication:

\begin{prop}
A normalizable Drinfeld module with CM by $\mathcal{O}_K$ can be defined over the Hilbert class field $H_K$.

\end{prop}
\begin{proof}
See Theorem 7.5.17 and Definition 7.5.18 in \cite{P23}
\end{proof}

\begin{rem}\label{repnw}

For each isomorphism class $[\phi]\in {\rm CM}(\mathcal{O}_K,\iota)$, we may choose a  representative $\phi$ is defined over $H_K$. Furthermore, We fix a place $\mathfrak{P}$ of $H_K$ above $\mathfrak{p}$ of $F$. And consider the local field $H_{K,\mathfrak{P}}$ of $H_K$ at $\mathfrak{P}$. Let $H_{K,\mathfrak{P}}^{nr}$ be the maximal unramified extension of $H_{K,\mathfrak{P}}$. We set ${\check{H}_{K,\mathfrak{P}}^{nr}}$ to be the completion of $H_{K,\mathfrak{P}}^{nr}$, and its discrete valuation ring is denoted by $W$. Let $\mu$ be a fixed uniformizer of $W$ with normalized valuation $\nu$. Therefore, up to isomorphisms, we may assume the representative $\phi$ of $[\phi]\in {\rm CM}(\mathcal{O}_K,\iota)$ is defined over $W$. 

\end{rem}

\begin{conv}\label{wn}
The notation $W_n$ in Theorem \ref{intro2} (resp. Theorem \ref{intro1}) is defined to be $W/\mu^n W$ (resp. $W/\mu^nW$ in Terminology \ref{phi'}). 
\end{conv}

\subsection{Connection to counting isomorphisms}

Now we are able to define the singular moduli for higher rank Drinfeld modules. Fix a $J$-invariant $J^{(\delta_1,\cdots,\delta_{r-1})}$. Consider the product

$$J_{\mathcal{O}_K}^{(\delta_1,\cdots,\delta_{r-1})}:=\prod_{[\phi]\in {\rm CM}(\mathcal{O}_K,\iota)}J^{(\delta_1,\cdots,\delta_{r-1})}(\phi).$$

Our first goal is to study the number $\nu(J_{\mathcal{O}_K}^{(\delta_1,\cdots,\delta_{r-1})})$. 

\begin{thm}\label{thm1}

$$
\nu(J_{\mathcal{O}_K}^{(\delta_1,\cdots,\delta_{r-1})})\geqslant \frac{\sum_{i=1}^{r-1}\delta_i}{q^r-1}\sum_{[\phi]\in{\rm CM}(\mathcal{O}_K,\iota)}\sum_{n\geqslant 1} \#{\rm Iso}_{W/\mu^nW}(\phi, \varphi) $$

Here ${\rm Iso}_{W/\mu^nW}(\phi, \varphi)$ denotes the set of isomorphisms from $\phi$ to $\varphi$ after reduction modulo $\mu^n$.

\end{thm}

 Let us consider an isomorphism class $[\phi]\in {\rm CM}(\mathcal{O}_K,\iota)$ with representative $\phi$ defined over $W$. On the other hand, $\varphi_T=T+\tau^r$ which can also be defined over $W$. Fix a $J$-invariant $J^{(\delta_1,\cdots,\delta_{r-1})}$, we always have $J^{(\delta_1,\cdots,\delta_{r-1})}(\varphi)=0$. Thus

$$\nu(J^{(\delta_1,\cdots,\delta_{r-1})}(\phi))=\nu(J^{(\delta_1,\cdots,\delta_{r-1})}(\phi)-J^{(\delta_1,\cdots,\delta_{r-1})}(\varphi)).$$

Theorem \ref{thm1} can be proved immediately from the following lemma:

\begin{lem}\label{Jest}

$$
\nu(J^{(\delta_1,\cdots,\delta_{r-1})}(\phi))\geqslant \frac{\sum_{i=1}^{r-1}\delta_i}{q^r-1}\sum_{n\geqslant 1} \#{\rm Iso}_{W/\mu^nW}(\phi, \varphi)
$$

\end{lem}

\begin{proof}
If $\phi$ is not isomorphic to $\varphi$ after reduction modulo $\mu$, then $\phi \not\cong \varphi$ after reduction modulo $\mu^n$ for any $n\geqslant 1$. Hence the right hand side is equal to $0$. As the left hand side is non-negative, the inequality always holds.

Let us assume that $\phi$ is isomorphic to $\varphi$ after reduction modulo $\mu^k$ but they are not isomorphic after reduction modulo $\mu^{k+1}$. Since the residue field $W/\mu W$ is algebraically closed, we may apply Hensel's lemma to replace $\phi$ by an isomorphic copy over $W$ such that
$$\phi_T=T+g_1\tau+\cdots+g_{r-1}\tau^{r-1}+\tau^r.$$

Our assumption then implies $g_i\equiv 0 \mod \mu^k$ for all $1\leqslant i\leqslant r-1$, but some $ g_i\not\equiv 0 \mod \mu^{k+1}$.

We compute $J^{(\delta_1,\cdots,\delta_{r-1})}(\phi)=g_1^{\delta_1}\cdot\cdots\cdot g_{r-1}^{\delta_{r-1}}$. Since $g_i\equiv 0 \mod \mu^k$, we get
\begin{equation}
\nu(J^{(\delta_1,\cdots,\delta_{r-1})}(\phi))\geqslant k(\delta_1+\cdots+\delta_{r-1}).\tag{3}
\end{equation}

On the other hand, from our assumption that $ \phi\cong \varphi \mod \mu^i$ for  $1 \leqslant i \leqslant k$, we have 
$$\# {\rm Iso}_{W/\mu^iW}(\phi, \varphi)=\# {\rm Aut}_{W/\mu^i W}(\varphi)=\# \{c\in \bar{\F}_q \mid c^{q^r-1}=1\} =\# \F_{q^r}^*=q^r-1 \textrm{ for } 1 \leqslant i \leqslant k.$$ Thus we know that
\begin{equation}
\sum_{n\geqslant 1} \#{\rm Iso}_{W/\mu^nW}(\phi, \varphi)=k(q^r-1). \tag{4}
\end{equation}

Now we can combine (3) and (4) together. From the condition (1) in the definition of $J$-invariant (see section \ref{j}), we compute the smallest possible number $(\delta_1+\cdots+\delta_{r-1})$ for $\delta_i$'s satisfy the equation 
$$ \delta_1(q-1)+\delta_2(q^2-1)+\cdots+\delta_{r-1}(q^{r-1}-1)=\delta_r(q^r-1).$$

The smallest possible is to make $\delta_{r-1}$ as large as possible. If $\delta_r=1$, then we choose $\delta_{r-1}=q$, $\delta_1=1$, and other $\delta_i=0$. This gives us $$(\delta_1+\cdots+\delta_{r-1})\geqslant q+1.$$ If $\delta_r>1$, the smallest possible $\delta_i's$ would have $\delta_{r-1}=\delta_r\cdot q$. Thus we have $$(\delta_1+\cdots+\delta_{r-1})\geqslant \delta_{r-1}=\delta_r\cdot q>q+1.$$

Therefore, the inequality (3) implies

$$J^{(\delta_1,\cdots,\delta_{r-1})}(\phi)\geqslant k(\delta_1+\cdots+\delta_{r-1})\geqslant k(q+1).$$

Thus we have

$$
\begin{array}{lll}
\nu(J^{(\delta_1,\cdots,\delta_{r-1})}(\phi)) &\geqslant k(\sum_{i=1}^{r-1}\delta_i)=\frac{\sum_{i=1}^{r-1}\delta_i}{q^r-1}\cdot k(q^r-1)=\frac{\sum_{i=1}^{r-1}\delta_i}{q^r-1}\sum_{n\geqslant 1} \#{\rm Iso}_{W/\mu^nW}(\phi, \varphi)\\
\end{array}.$$

\end{proof}

Theorem \ref{thm1} is now a direct application of Lemma \ref{Jest}:

\begin{proof}[Proof of Theorem \ref{thm1}]
$$
\begin{array}{lll}

\nu(J_{\mathcal{O}_K}^{(\delta_1,\cdots,\delta_{r-1})})&=&\nu\left(  \prod_{[\phi]\in {\rm CM}(\mathcal{O}_K,\iota)}J^{(\delta_1,\cdots,\delta_{r-1})}(\phi)   \right)\\
\ \\
&=&\sum_{[\phi]\in{\rm CM}(\mathcal{O}_K,\iota)}\nu\left(J^{(\delta_1,\cdots,\delta_{r-1})}(\phi)-J^{(\delta_1,\cdots,\delta_{r-1})}(\varphi) \right).

\end{array}$$
Then apply Lemma \ref{Jest} on each summand in the above equation.
\end{proof}

Unlike the rank-$2$ case, it is still possible to have that $\mu^i$ divides $J^{(\delta_1,\cdots,\delta_{r-1})}(\phi)$ even $\phi\not\cong \varphi \mod \mu^i$ for some $i\geqslant 1$. As our estimation only takes into account the valuation contributed from isomorphism after reduction modulo power of $\mu$, the right hand side of inequalities (1) or (2)  is certainly smaller than $\nu(J_{\mathcal{O}_K}^{(\delta_1,\cdots,\delta_{r-1})})$. One might be curious on when we can reach the equalites in Theorem \ref{thm1}, we conclude the conditions in the corollary below:

\begin{cor}\label{cjest}
The equality  in Lemma \ref{Jest} happens when the following condition holds:

\begin{enumerate}
\item[] For the chosen basic $J$-invariant $J^{(\delta_1,\cdots,\delta_{r-1})}$, if $\phi\not\cong\varphi \mod \mu^i$ for some $i\geqslant 1$, then $g_j\not\equiv 0 \mod \mu^i$ for all $1\leqslant j\leqslant r-1$ such that $\delta_j\neq0$.
\end{enumerate}

\end{cor}

\begin{proof}
Immediate from the proof of Lemma \ref{Jest}

\end{proof}

Hence we can describe when will the inequality in Theorem \ref{thm1} becomes an equality. 

\begin{cor}\label{corthm1}

The equality  in Theorem \ref{thm1} happens when the condition in Corollary \ref{cjest} is satisfied for each representative $\phi$ of $[\phi]\in\textrm{CM}(\mathcal{O}_K,\iota)$ 
\end{cor}

However, Corollary \ref{corthm1} involves in explicit coefficients of Drinfeld modules with CM by $\mathcal{O}_K$. In the computational perspective, it is hard to list all the Drinfeld modules having CM by $\mathcal{O}_K$ up to isomorphism over $\mathbb{C}_\infty$.

\begin{note}
In the end of the subsection, we generalized Lemma \ref{Jest} to the estimation of the valuation $\nu$ of the difference 
$$J^{(\delta_1,\cdots,\delta_{r-1})}(\phi)-J^{(\delta_1,\cdots,\delta_{r-1})}(\phi')$$
for arbitrary two rank-$r$ Drinfeld modules over $W$ with good reduction.  We formulate a rather coarse estimation in the proposition below. \end{note}

\begin{prop}\label{genjest}
Let $W$ be a complete discrete valuation ring with valution $\nu$ and uniformizer $\mu$. And suppose that the residue field $W/\mu W$ is algebraically closed.  Let $\phi$ and $\phi'$ be two rank-$r$ Drinfeld module over $W$ both have good reduction. Let $J^{(\delta_1,\cdots,\delta_{r-1})}$ be a $J$-invariant defined in section 2.3. Then

$$\nu(J^{(\delta_1,\cdots,\delta_{r-1})}(\phi)-J^{(\delta_1,\cdots,\delta_{r-1})}(\phi'))\geqslant \frac{1}{q^r-1}\sum_{n\geqslant 1} \#{\rm Iso}_{W/\mu^nW}(\phi, \phi').$$

\end{prop}
\begin{proof}
If $\phi$ and $\phi'$ are isomorphic over $W$, then both sides of the inequality are infinity, done. If $\phi$ and $\phi'$ are not isomorphic after reduction modulo $\mu$, then the right hand side is equal to $0$ and the inequality always hold. Therefore, we may assume that $\phi$ and $\phi'$ are not isomorphic over $W$, but they are isomorphic after reduction modulo $\mu$.

Now we may assume that $\phi$ is isomorphic to $\phi'$ after reduction modulo $\mu^k$ for some $k\in \mathbb{Z}_{\geqslant1}$, but they are not isomorphic after reduction modulo $\mu^{k+1}$. Since $W/\mu W$ is algebraically closed, Hensel's lemma implies that we may normalize the leading coefficient of $\phi_T$ and $\phi'_T$ by replacing with isomorphic copies over $W$ if necessary. Thus we may write 

$$\phi_T=T+g_1\tau+\cdots+g_{r-1}\tau^{r-1}+\tau^r, \textrm{ and } \phi'_T=T+g_1'\tau+\cdots+g_{r-1}'\tau^{r-1}+\tau^r.$$
Here $g_i$ and $g_i'$ are in $W$ for all $1\leqslant i\leqslant r-1$. Since $\phi$ and $\phi'$ are isomorphic modulo $\mu^k$ but not isomorphic modulo $\mu^{k+1}$, there is some $c\in \F_{q^r}^*\subset W$ such that

$$\Bigg\{
\begin{array}{cc}
c^{q^i-1}g_i\equiv g_i'\mod \mu^k & \textrm{ for any } 1\leqslant i\leqslant r-1\\
\ \\
c^{q^i-1}g_i\not\equiv g_i'\mod \mu^{k+1} & \textrm{ for some } 1\leqslant i\leqslant r-1
\end{array}.$$

Thus we can write $g_i'=c^{q^i-1}g_i+\mu^k\cdot u_i$ for any $1\leqslant i\leqslant r-1$. Here $u_i\in W$ for all $i$, and there is some $u_j$ belongs to the unit group $W^*$ of $W$. 

Hence we have the following cases:
\begin{enumerate}
\item[case 1.] $g_1^{\delta_1}\cdot\cdots\cdot g_{r-1}^{\delta_{r-1}}= 0 $  in $W$. Then we let separate the set $\{1,2,\cdots, r-1\}$ into 
$$\mathcal{A}:=\{1\leqslant i_\ell \leqslant r-1\mid g_{i_\ell}=0\} \textrm{ and } \mathcal{B}:=\{1,\cdots, r-1\}-\mathcal{A}.$$
$$
\begin{array}{lll}
\nu(J^{(\delta_1,\cdots,\delta_{r-1})}(\phi)-J^{(\delta_1,\cdots,\delta_{r-1})}(\phi'))\\
\ \\
=\nu[(g_1^{\delta_1}\cdot\cdots\cdot g_{r-1}^{\delta_{r-1}})-(g_1'^{\delta_1}\cdot\cdots\cdot g_{r-1}'^{\delta_{r-1}})]\\
\ &\ &\ \\
=\nu[(c^{q-1}g_1+\mu^ku_1)^{\delta_1}\cdot\cdots\cdot(c^{q^{r-1}-1}g_{r-1}+\mu^ku_{r-1})^{\delta_{r-1}} ]\\
\ &\ &\ \\
\geqslant k\cdot(\sum_{j\in \mathcal{A}}\delta_j)
\end{array}.
$$

\item[case 2.] $g_1^{\delta_1}\cdot\cdots\cdot g_{r-1}^{\delta_{r-1}}\neq 0$ in $W$. Then we let separate the set $\{1,2,\cdots, r-1\}$ into 
$$\mathcal{A}:=\{1\leqslant i_\ell \leqslant r-1\mid \delta_\ell\neq0\} \textrm{ and } \mathcal{B}:=\{1,\cdots, r-1\}-\mathcal{A}.$$
Thus we also have $g_\ell\neq 0$ for any $\ell\in \mathcal{A}$. Now we can compute

$$
\begin{array}{lll}
\nu(J^{(\delta_1,\cdots,\delta_{r-1})}(\phi)-J^{(\delta_1,\cdots,\delta_{r-1})}(\phi'))=\nu[(g_1^{\delta_1}\cdot\cdots\cdot g_{r-1}^{\delta_{r-1}})-(g_1'^{\delta_1}\cdot\cdots\cdot g_{r-1}'^{\delta_{r-1}})]\\
\ &\ &\ \\
=\nu[(1-c^{\sum_{1\leqslant j\leqslant r-1} (q^j-1)\delta_j})(g_1^{\delta_1}\cdot\cdots\cdot g_{r-1}^{\delta_{r-1}})-\mu^k(\sum_{\ell\in \mathcal{A}}\delta_\ell u_\ell(c^{q^\ell-1}g_\ell)^{\delta_\ell-1})-\textrm{ higher $\mu$-power terms}]&&\\
\ &\ &\ \\
\geqslant k
\end{array}.
$$
Here the term $(1-c^{\sum_{1\leqslant j\leqslant r-1} (q^j-1)\delta_j})(g_1^{\delta_1}\cdot\cdots\cdot g_{r-1}^{\delta_{r-1}})$ is equal to zero because of condition (1) in page 3 and the fact that $c\in\F_{q^r}^*$.

\end{enumerate}

On the other hand, since $\phi\cong \phi' \mod \mu^i$ for $1\leqslant i\leqslant k$ and $\phi\not\cong \phi' \mod \mu^{k+1}$, we get

$$\# {\rm Iso}_{W/\mu^iW}(\phi, \phi')=\# {\rm Aut}_{W/\mu^i W}(\phi)=\#\{ \F_{q^r}^*\bigcap_{1\leqslant j\leqslant r-1} \F_{q^j}^*\mid g_j\not\equiv 0\mod \mu^i\} \textrm{ for } 1 \leqslant i \leqslant k.$$

Therefore, $\# {\rm Iso}_{W/\mu^iW}(\phi, \phi')\leqslant q^r-1$ for any $1\leqslant i\leqslant k$. This implies 
$$\sum_{n\geqslant 1} \#{\rm Iso}_{W/\mu^nW}(\phi, \phi')= \sum_{i=1}^k g.c.d.(q^r-1, q^j-1 \mid g_j\not\equiv 0 \mod \mu^i \textrm{ for } 1\leqslant j\leqslant r-1)\leqslant k\cdot (q^r-1)$$

Combining our estimation on $\nu(J^{(\delta_1,\cdots,\delta_{r-1})}(\phi)-J^{(\delta_1,\cdots,\delta_{r-1})}(\phi'))$ and $\sum_{n\geqslant 1} \#{\rm Iso}_{W/\mu^nW}(\phi, \phi')$, we achieve

$$\nu(J^{(\delta_1,\cdots,\delta_{r-1})}(\phi)-J^{(\delta_1,\cdots,\delta_{r-1})}(\phi'))\geqslant \frac{1}{q^r-1}\sum_{n\geqslant 1} \#{\rm Iso}_{W/\mu^nW}(\phi, \phi').$$

\end{proof}

\begin{ex}
Let $W$ be as in Proposition \ref{genjest}. We set $\phi$ to be the sparse Drinfeld module $\phi_T=T+\tau^{j}+\tau^r$ with $1\leqslant j \leqslant r-1$. Consider any rank-$r$ Drinfeld module $\phi'$ with good reduction, and $J$-invariant $J^{(\delta_1,\cdots,\delta_{r-1})}$. Suppose that $\phi\cong \phi' \mod \mu^k$ but $\phi\not\cong\phi' \mod \mu^{k+1}$, then we have
$$\Bigg\{
\begin{array}{cc}
\nu(J^{(\delta_1,\cdots,\delta_{r-1})}(\phi)-J^{(\delta_1,\cdots,\delta_{r-1})}(\phi'))\geqslant k\cdot(\sum_{1\leqslant \ell\neq j \leqslant r-1}\delta_{\ell})\\
\ \\
\sum_{n\geqslant 1} \#{\rm Iso}_{W/\mu^nW}(\phi, \phi')= \sum_{i=1}^k g.c.d.(q^r-1, q^{j}-1)=k\cdot g.c.d.(q^r-1, q^{j}-1)\end{array}.$$

Therefore, we can refine the inequality in Proposition \ref{genjest} into 
$$\nu(J^{(\delta_1,\cdots,\delta_{r-1})}(\phi)-J^{(\delta_1,\cdots,\delta_{r-1})}(\phi'))\geqslant \frac{\sum_{1\leqslant \ell\neq j \leqslant r-1}\delta_{\ell}}{g.c.d.(q^r-1, q^{j}-1)}\sum_{n\geqslant 1} \#{\rm Iso}_{W/\mu^nW}(\phi, \phi').$$

\end{ex}
\subsection{CM-lifting}

In order to compute $\nu(J_{\mathcal{O}_K}^{(\delta_1,\cdots,\delta_{r-1})})$, Theorem \ref{thm1} tells us that it is enough to count the number of isomorphisms between ``$\phi\mod \mu^n$'' and ``$\varphi\mod \mu^n$'' for each $[\phi]\in{\rm{CM}}(\mathcal{O}_K,\iota)$. In this subsection, we apply deformation theory for Drinfeld modules with level structure to study $\#{\rm Iso}_{W/\mu^nW}(\phi, \varphi)$. Then we can turn the number of isomorphisms between a CM Drinfeld module $\phi$ and $\varphi$ into counting the number of certain elements in the endomorphism ring of $\varphi_T=T+\tau^r$.

Since $\phi$ over $W$ has CM by $\mathcal{O}_K$. We may view $\phi$ as a Drinfeld $\mathcal{O}_K$-module of rank $1$ over $W$. Thus $\phi_{s_i}\in {\rm End}_W(\phi)$ is well-defined for any $1\leqslant i\leqslant t$. By taking reduction modulo $\mu^n$, we get $\bar{\phi}_{s_i}\in {\rm End}_{W/\mu^nW}(\phi)$.

\begin{defi}
Given $w_1, w_2 \in {\rm Iso}_{W/\mu^nW}(\varphi,\phi)$, we define $w_1\sim w_2$ if and only if there is some $\xi\in\bar{\F}_{q}^*$ such that $w_2=\xi w_1$ and $$\xi^{-1}\bar{\phi}_{a}\xi=\bar{\phi}_{a} \textrm{ for } a\in \mathcal{O}_K.$$

 It is clear that $\sim$ is an equivalence relation. Hence we can divide the disjoint union $\sqcup_{[\phi]\in {\rm CM}(\mathcal{O}_K,\iota)}{\rm Iso}_{W/\mu^nW}(\varphi,\phi)$ into equivalence classes with respect to $\sim$.
\end{defi}

\begin{thm}\label{thm2}
Given a positive integer $n$, there is a one-to-one correspondence between equivalence classes of $$\sqcup_{[\phi]\in {\rm CM}(\mathcal{O}_K,\iota)}{\rm Iso}_{W/\mu^nW}(\varphi,\phi) \textrm{ with respect to }\sim,$$ and the set  
{\center$S_n:=$ $\Bigg\{\eta:\mathcal{O}_K\hookrightarrow {\rm End}_{W/\mu^nW}(\varphi)\ \bigg|$
$
\begin{array}{cc}
 \eta \textrm{ is an $A$-algebra embedding},  \\
 \newline \eta|_A=\varphi\mod\mu^n, \textrm{and } \partial\circ\eta\equiv \iota \mod \mu^n  \\
\end{array}
$
$\Bigg\}.$}
\end{thm}

\begin{proof}

For any $w\in {\rm Iso}_{W/\mu^nW}(\varphi,\phi)$, we consider the composition of maps:

$$\eta_w:=w^{-1}\circ \bar{\phi}\circ w \in {\rm End}_{W/\mu^n W}(\varphi )$$

It is clear that the map $\eta_{w}$ satisfies the conditions $\eta|_A=\varphi\mod \mu^n$ and $\partial\circ\eta_w=\partial\circ \bar{\phi}\equiv \iota\mod \mu^n.$  Moreover, if there is some $w'\in {\rm Iso}_{W/\mu^nW}(\varphi, \phi)$ such that $w'\sim w$, then we have $w'=\xi w$ for some $\xi\in\bar{\F}_{q}^*$ with $\xi^{-1}\bar{\phi}_{s_i}\xi=\bar{\phi}_{s_i}$. Thus we have
$$\eta_{w'}=w^{-1}\circ \xi^{-1}\bar{\phi}\xi \circ w=w^{-1}\circ \bar{\phi}\circ w=\eta_{w}.$$
Therefore, we constructed a well-defined map $$[w]\in \textrm{ equivalent classes of } \sqcup_{[\phi]\in {\rm CM}(\mathcal{O}_K,\iota)}{\rm Iso}_{W/\mu^nW}(\varphi,\phi) \mapsto \eta_w\in S_n.$$ 

Now we prove the map is indeed surjective. Let us start from a given element $\eta\in S_{n}$. The first goal is to prove the following:  \begin{claim*}
$(\bar{\varphi}\equiv \varphi \mod \mu^n,\eta)$ over $W/\mu^n W$ can be lifted, uniquely up to $W$-isomorphism, to $(\psi,\tilde{\eta})$ over $W$, where 

\begin{enumerate}
\item[(a)]

$\psi$ is a rank-$r$ Drinfeld module over $W$ with CM by $\mathcal{O}_K$.

\item[(b)]
$\psi\equiv \varphi \mod \mu^n$.

\item[(c)]
$\tilde{\eta}:\mathcal{O}_K\hookrightarrow {\rm End}_W(\psi)$ is an $A$-algebra embedding, and $\tilde{\eta}\equiv \eta \mod \mu^n$ for all.
\end{enumerate}

\end{claim*}
\begin{proof}[Proof of Claim]

The main idea is inspired from the statements between Theorem 2.8 and 2.9 in page 6 of \cite{BB21}. Let us follow their use of notation. 
\begin{enumerate}
\item[$\bullet$]
$\fq=\mu W \bigcap \mathcal{O}_K$ a prime ideal of $\mathcal{O}_K$.
\item[$\bullet$]
$\mathcal{O}_\fq=$ completion of $\mathcal{O}_K$ at $\fq$. Fix a generator of $\fq\mathcal{O}_{\fq}$ to be $\omega$.

\item[$\bullet$]
$\check{\mathcal{O}}:=W$
\item[$\bullet$]
$\check{k}:=\check{\mathcal{O}}/\mu \check{\mathcal{O}}$
\end{enumerate}
We begin with the case $n=1$. Denote  $\underline{\varphi}$ by the reduction of ${\varphi}$ modulo $\mu$, and $\underline{\eta}\in S_1$. Then we may view $\underline{\varphi}$ as a rank-$1$ Drinfeld $\mathcal{O}_K$-module over $\check{k}$ by setting $$\underline{\varphi}_a=\underline{\eta}(a) \textrm{ for all } a\in\mathcal{O}_K.$$
Now we use Drinfeld's Deformation theory to lift $\underline{\varphi}$, the rank-$1$ Drinfeld $\mathcal{O}_K$-module over $\check{k}$, to a rank-$1$ Drinfeld $\mathcal{O}_K$-module over $\check{O}.$  By the Drinfeld module analogue of Serre-Tate theorem (\cite{D74}, proposition 5.4 or \cite{BB21}, Theorem 2.9), it is enough to construct a lifting of $\underline{\varphi}[\fq^{\infty}]$, the $\mathcal{O}_\fq$-divisible group of $\underline{\varphi}$, to $\check{O}$. By proposition 2.6 of \cite{BB21} and the fact that  $\underline{\varphi}$ is a rank-$1$ Drinfeld $\mathcal{O}_K$-module, the lifting problem becomes to construct a lifting of a height-$1$ formal $\mathcal{O}_\fq$-module corresponding to the $\mathcal{O}_\mathfrak{q}$-divisible group $\underline{\varphi}[\fq^\infty]$. Now for such a height-$1$ formal $\mathcal{O}_\mathfrak{q}$-module over $\check{k}$, we always have a lifting to $\check{\mathcal{O}}$, unique up to $\check{O}$-isomorphism. The existence of lifting is due to the existence of universal deformation of formal $\mathcal{O}_\fq$-modules, see Proposition 4.2 in \cite{D74}. The proposition also implies the uniqueness of lifting up to $\check{\mathcal{O}}$-isomorphism because our formal $\mathcal{O}_\fq$-module here has height $1$.   Thus we get a lifting $(\psi,\tilde{\eta})$ of $(\underline{\varphi},\underline{\eta})$ to $\check{\mathcal{O}}$. Note that in the above CM-lifting process we do not assume $\varphi\mod \mu$ has supersingular reduction.

For general $n\geqslant 2$, Drinfeld's deformation theory still works when replacing $\check{\mathcal{O}}$ by any Noetherian local $\check{\mathcal{O}}$-algebra. And we know that $W/\mu^n W$ is a Noetherian local $\check{\mathcal{O}}$-algebra, so we can run through the same process as in the case $n=1$ while replacing $\check{\mathcal{O}}=W$ by $\check{\mathcal{O}}=W/\mu^nW$. Indeed, we start with a pair $(\bar{\varphi},\eta)$ defined over $W/\mu^n W$. Consider the pair under reduction modulo $\mu$, we get $(\underline{\varphi}, \underline{\eta})$ defined over $\check{k}$. Running through the process as in the case $n=1$, we obtained a pair $(\psi',\eta')$ defined over $\check{\mathcal{O}}=W/\mu^n W$. Moreover, because such a lifting is unique up to $\check{\mathcal{O}}$-isomorphism, we have the following diagram over $W/\mu^n W$:
$$
\begin{array}{ccc}
\bar{\varphi}&\xrightarrow{\eta(a)}&\bar{\varphi}\\
 \ \ \downarrow{\simeq}&&\ \ \downarrow{\simeq}\\
\psi'&\xrightarrow{\eta'(a)}&\psi'
\end{array}
$$
Therefore, a lifting of $(\underline{\varphi},\underline{\eta})$ over $W/\mu W$ to $(\psi,\tilde{\eta})$ over $W$ can also be viewed as a lifting of $(\bar{\varphi},\eta)$ over $W/\mu^n W$ for $n\geqslant 2$.

It remains to check conditions (a), (b), and (c). For (a), we know that $\psi$ is a rank-$1$ Drinfeld $\mathcal{O}_K$-module over $W$. This implies $\psi|_A$ is a rank-$r$ Drinfeld $A$-module over $W$ with CM by $\mathcal{O}_K$. 
On the other hand, the conditions (b) and (c) are clear from the deformation process described in previous paragraph

\end{proof}

Now $\psi$ is a normalizable rank-$r$ Drinfeld $A$-module over $W$ with CM by $\mathcal{O}_K$. From \cite{D74}, Proposition 5.3 that moduli space $M^1_I$ of isomorphic classes of rank-$1$ Drinfeld-$\mathcal{O}_K$ modules together with a level $I$ structure is a fine moduli space when $0\neq I\subset \mathcal{O}_K$ is an admissible ideal of $\mathcal{O}_K$, i.e. $V(I)$ contains more than one point. Since a fine moduli scheme defined over a local field can be viewed as the base change of the fine moduli scheme defined over a global field. We may say, up to $W$-isomorphism, that $\psi$ is defined globally. From the assumption that $K/F$ is normal, any rank-$r$ Drinfeld module over $\mathbb{C}_\infty$ has CM by $\mathcal{O}_K$ is normalizable (cf. \cite{P23}, Definition 7.5.6). Hence  $[\psi]\in {\rm CM}(\mathcal{O}_K,\iota)$. Therefore, $\psi$ lies in some isomorphism class $[\phi]\in {\rm CM}(\mathcal{O}_K,\iota)$, and recall from Remark \ref{repnw} that the representative $\phi$ is made to be defined over $W$.  Thus we have an isomorphism $\beta:\psi\xrightarrow{\sim} \phi$ defined over $W$ such that $$\tilde{\eta}(a)=\beta^{-1}\circ \phi_{a}\circ\beta \textrm{ for all } a\in\mathcal{O}_K.$$ This implies $\bar{\beta}\equiv \beta \mod \mu^n$ lies in ${\rm Iso}_{W/\mu^n W}(\varphi,\phi)$ and $\eta_{\bar{\beta}}=(\tilde{\eta}\mod \mu^n)=\eta$,  which shows that the map $w\mapsto \eta_w$ is surjective.

On the other hand, if there are equivalence classes $[w_1]$ and $[w_2]$ such that $\eta_{w_1}=\eta_{w_2}$. Then $[w_1]$ and $[w_2]$ both lie in the collection of equivalent classes of  ${\rm Iso}_{W/\mu^nW}(\varphi,\phi)$ for some $[\phi]\in {\rm CM}(\mathcal{O}_K,\iota)$ because of the uniqueness, up to $W$-isomorphism, of CM-lifting that we proved in the claim. Now $w_1w_2^{-1}$ is an isomorphism from $\phi \mod \mu^n$ to itself such that $$w_2w_1^{-1}\circ \bar{\phi}_{a}\circ w_1w_2^{-1}=\bar{\phi}_{a} \textrm{ for all }a\in\mathcal{O}_K.$$
The equality forces $w_1w_2^{-1}=\xi \in\bar{\F}_q^*$ satisfies $\xi^{-1}\bar{\phi}_{a}\xi=\bar{\phi}_{a}$ for all $a\in\mathcal{O}_K$,  so $[w_1]=[w_2]$. Therefore, the map 
$$[w]\mapsto \alpha_{i,w}$$
is injective. This completes the proof of the theorem.
\end{proof}

\begin{rem}
In the  proof of claim in Theorem \ref{thm2}, we view $\varphi \mod \mu$ together with $\underline{\eta}\in S_1$ as a rank-$1$ Drinfeld $\mathcal{O}_K$-module $\underline{\varphi}$. In order to apply Drinfeld's deformation theory, the requirement on ``Drinfeld $\mathcal{O}_K$-module'' is necessary. Indeed, if we wish to find CM-lifting for a rank-$1$ Drinfeld $\mathcal{O}$-module $\underline{\varphi}$ over $W/\mu W$ where $\mathcal{O}$ is an order in $\mathcal{O}_K$. Then we need to find an $W/\mu W$-isogenous Drinfeld $\mathcal{O}_K$ module $\underline{\psi}$ over $W/\mu W$ (here an extra condition on the conductor of $\mathcal{O}\subset \mathcal{O}_K$ is needed), and apply Drinfeld's deformation theory to $\underline{\psi}$. After obtaining a CM-lifting of $\underline{\psi}$ to a  rank-$1$ Drinfeld $\mathcal{O}_K$-module over $W$, we then construct an $W$-isogeny to a rank-$r$ Drinfeld $A$-module $\phi$ with endomorphism ring ${\rm End}(\phi)\supset \mathcal{O}$, and the reduction of $\phi$ modulo $\mu$ is exactly $\varphi \mod \mu$. However, the Drinfeld module $\phi$ is constructed via an isogeny, it is unclear whether we have ${\rm End}(\phi)= \mathcal{O}$. We refer the detail of this process to \cite{CP15}, Theorem 22.

\end{rem}

\begin{cor}\label{sn}
$$\#\sqcup_{[\phi]\in {\rm CM}(\mathcal{O}_K,\iota)}{\rm Iso}_{W/\mu^nW}(\varphi,\phi)=\sum_{[\phi]\in {\rm CM}(\mathcal{O}_K,\iota)}\#{\rm Iso}_{W/\mu^nW}(\phi,\varphi)\geqslant (q-1)\#S_n.$$
\end{cor}

\begin{proof}
From Theorem \ref{thm2}, we have 
$$\#S_n=\textrm{ number of equivalence classes of } \sqcup_{[\phi]\in {\rm CM}(\mathcal{O}_K,\iota)}{\rm Iso}_{W/\mu^nW}(\varphi,\phi).$$

On the other hand, for each equivalence class $[w]$, it contains at least those $\xi w$ where $\xi\in \F_q^*$. As total, there are at least $(q-1)\# S_n$ many elements in ${\rm Iso}_{W/\mu^nW}(\phi, \varphi)$. 
\end{proof}

\begin{rem}\label{rk2}
When rank $r=2$, $q$ is odd, and $\mathcal{O}_K=A[s]$ where $s=\sqrt{d}$ for some squarefree $d\in A$. There is only one basic $J$-invariant $j=J^{q+1}=\frac{g_1^{q+1}}{\Delta}$.  The $j$-invariant can distinguish Drinfeld modules up to isomorphism, so Theorem \ref{thm1} can be reduced into the following equality:
 $$\nu(J_{\mathcal{O}_K}^{(q+1)})= \frac{q+1}{q^2-1}\sum_{[\phi]\in{\rm CM}(\mathcal{O}_K,\iota)}\sum_{n\geqslant 1} \#{\rm Iso}_{W/\mu^nW}(\phi, \varphi). $$

Moreover, as stated in arguments after (5.4) of \cite{D91}, there are exactly $q-1$ many $\xi\in\bar{\F}_q^*$ satisfying $\xi^{-1}\bar{\phi}_a\xi=\bar{\phi}_a$ for all $a\in \mathcal{O}_K$. Hence we get
$$\#\sqcup_{[\phi]\in {\rm CM}(\mathcal{O}_K,\iota)}{\rm Iso}_{W/\mu^nW}(\varphi,\phi)=(q-1)\cdot \#S_n.$$
Reducing the valuation back to $\fp$, we have
$${\rm ord}_\fp(J_{\mathcal{O}_K}^{(q+1)})=\frac{1}{e_{K,\fp}}\sum_{n\geqslant 1}\#S_n,$$
where $e_{K,\fp}$ is the ramification index of $\fp$ in $K$. This equality is exactly equation (5.5) in \cite{D91}.
\end{rem}

Next, we make the following intepretation on $\#S_n$:

\begin{prop}\label{algemb}
Define the set $\mathcal{M}_n$ to be
{\center$\Bigg\{\eta:\mathcal{O}_K\hookrightarrow {\rm End}_{W/\mu^nW}(\varphi)\ \bigg|$
$
\begin{array}{cc}
 \eta \textrm{ is an $A$-algebra embedding }, \textrm{ and } \eta|_A=\varphi\mod\mu^n 
 \end{array}
$
$\Bigg\}.$}

Then we have
$$\#S_n\geqslant\frac{1}{\rsep}\# \mathcal{M}_n.$$

\end{prop}

\begin{proof}
Given $\eta\in \mathcal{M}_n$, the composition $\partial\circ \eta:\mathcal{O}_K\rightarrow W/\mu^nW$ induces an $A$-algebra embedding into $W/\mu^nW$. Since $K/F$ is normal, the field extension  $H_{K,\mathfrak{P}}^{nr}/F_\fp$ is normal. For simplicity, we denote $G_\fp:={\rm Aut}(H_{K,\mathfrak{P}}^{nr}/F_\fp)$. We define an action of $G_\fp$ on $W/\mu^n W$ by $$\bar{a}^\sigma=\overline{a^\sigma} \textrm{ for }a\in W, \textrm{ and } \sigma\in G_\fp.$$
The $G_\fp$-action on $W/\mu^nW$ is well-defined since the action of automorphism group preserves valuation $|\cdot|_\fp$. Moreover, for $\eta(a)=\alpha_0+\alpha_1\tau_1+\cdots+\alpha_d\tau^d \in W/\mu^n W\{\tau\}$ where $a\in\mathcal{O}_K$, define $G_\fp$-action on $W/\mu^nW\{\tau\}$ by $\eta^\sigma(a):=\alpha_0^\sigma+\alpha_1^\sigma\tau+\cdots+\alpha_d^\sigma\tau^d$. 

For an $A$-algebra embedding $\partial\circ \eta:\mathcal{O}_K\rightarrow W/\mu^nW$, we lift to an $A$-algebra embedding $\widehat{\partial\circ \eta}:\mathcal{O}_K\rightarrow W$. There is an element $\sigma\in G_\fp$ such that $(\widehat{\partial\circ \eta})^{\sigma}=\iota.$ Then we can deduce $$(\partial\circ \eta)^{\sigma}=\iota \mod \mu^n.$$

 Now we separate $\mathcal{M}_n$ into disjoint subsets $\{\eta\in\mathcal{M}_n\mid \partial\circ \eta=\zeta\mod \mu^n\}$, each collects $\eta\in\mathcal{M}_n$ with the same embedding $\partial\circ \eta$, here $\zeta:\mathcal{O}_K\rightarrow W$ is an $A$-algebra embedding from $\mathcal{O}_K$ to $W$. There are at most $\rsep=[K:F]_{\rm sep}$ many such subsets. For any subset $$\{g\in \mathcal{M}_n\mid \partial\circ g=\zeta\mod \mu^n\}\neq\{f\in \mathcal{M}_n\mid \partial\circ f=\iota\mod \mu^n\},$$
one can construct an element $\sigma\in G_\fp$ from the previous paragraph such that 
$$(\partial\circ g)^\sigma=\iota\mod\mu^n.$$
Hence we can construct an injective map  
$$
\begin{array}{cc}
\sigma:&\{g\in \mathcal{M}_n\mid \partial\circ g=\zeta\mod \mu^n\} \longrightarrow \{f\in \mathcal{M}_n\mid \partial\circ f=\iota\mod \mu^n\}\\
& g\mapsto g^\sigma
\end{array}
$$
Also, $\sigma^{-1}$ induces an injective map from $\{f\in \mathcal{M}_n\mid \partial\circ f=\iota\mod \mu^n\}$ to $\{g\in \mathcal{M}_n\mid \partial\circ g=\zeta\mod \mu^n\}$. This proves that all such subsets have the same cardinality.

Thus we have
$$\#S_n=\#\{f\in \mathcal{M}_n\mid \partial\circ f=\iota \mod \mu^n\}\geqslant\frac{1}{\rsep}\#\mathcal{M}_n.$$

\end{proof}

\begin{rem}\label{snequal}
From the proof of Proposition \ref{algemb}(ii), we can say more explicitly that

$$S_n=\frac{\# \M_n}{\# \textrm{ of disjoint subsets }\{f\in\mathcal{M}_n\mid \partial\circ f=\zeta\mod \mu^n\} \textrm{ in }\M_n}$$
\end{rem}

Now we are in the position to state and prove our main result:

\begin{thm}\label{main1}
Let $r\geqslant 2$ be an integer, and $q=p^e$ be a prime power. Let $A=\F_q[T]$ be the polynomial ring, and $F=\F_q(T)$ be its fractional field. Let $K/F$ be an imaginary extension of degree $r$, with $\rsep=[K:F]_{\rm sep}$ defined to be its separable degree, and the integral closure of $A$ in $K$ is denoted by  $\mathcal{O}_K$. Furthermore, Assume $K$ 
is a normal extension over $F$. Fix a prime ideal $\fp$ of $A$. Then for any basic $J$-invariant $J^{(\delta_1,\cdots,\delta_{r-1})}$, we have

$$
\begin{array}{ccc}
{\rm{ord}}_\fp(J_{\mathcal{O}_K}^{(\delta_1,\cdots,\delta_{r-1})})&:=&{\rm{ord}}_\fp\left(\prod_{[\phi]\in {\rm CM}(\mathcal{O}_K,\iota)}J^{(\delta_1,\cdots,\delta_{r-1})}(\phi)\right)\nonumber\\
\ \\ 
&\geqslant& \frac{(\sum_{i=1}^{r-1}\delta_i)(q-1)}{\rsep\cdot(q^r-1)\cdot e_{K,\fp}}\sum_{n\geqslant 1} \#\mathcal{M}_n \\
\end{array}
.$$
Here $e_{K,\fp}$ is the ramification index of $\fp$ in $K$, and $ \mathcal{M}_n$ is the set of $A$-algebra embeddings $\eta:\mathcal{O}_K\hookrightarrow {\rm End}_{W/\mu^nW}(\varphi)$ such that $\eta|_A=\varphi\mod \mu^n$,  where $\varphi$ is the Drinfeld module $\varphi_T=T+\tau^r$.

\end{thm}
\begin{proof}
Immediate from Theorem \ref{thm1}, Corollary \ref{sn}, Proposition \ref{algemb}, and the fact that $${\rm{ord}}_\fp(J_{\mathcal{O}_K}^{(\delta_1,\cdots,\delta_{r-1})})=\frac{1}{e_{K,\fp}}\cdot \nu(J_{\mathcal{O}_K}^{(\delta_1,\cdots,\delta_{r-1})}).$$
\end{proof}

\begin{rem}
From condition (1) and (2) in the definition of $J^{(\delta_1,\cdots, \delta_{r-1})}$, we can deduce that $\sum_{i=1}^{r-1}\delta_i\geqslant q+1$. Hence the inequality in Theorem \ref{main1} can be further modified into 
$${\rm{ord}}_\fp(J_{\mathcal{O}_K}^{(\delta_1,\cdots,\delta_{r-1})})\geqslant \frac{(q+1)(q-1)}{\rsep\cdot(q^r-1)\cdot e_{K,\fp}}\sum_{n\geqslant 1} \#\mathcal{M}_n.$$

\end{rem}

\begin{rem}

From here to the end of section 3, we generalize Theorem \ref{main1} to the valuation

$${\rm ord}_\fp\left(\prod_{[\phi]\in{\rm CM}(\mathcal{O}_K,\iota)}J^{(\delta_1,\cdots,\delta_{r-1})}(\phi)-J^{(\delta_1,\cdots,\delta_{r-1})}(\phi')\right).$$
Here $\phi'$ is a fixed Drinfeld module of rank $r$ defined over a normal extension $K'/F$ that has good reduction at a place $\fq'$ of $K'$ stands above $\fp$.

\end{rem}

\begin{term}\label{phi'}
\begin{enumerate}
\item[$\bullet$] 
Let $r\geqslant 2$ be an integer, and $q=p^e$ be a prime power coprime to $r$. Let $A=\F_q[T]$ be the polynomial ring, and $F=\F_q(T)$ be its fractional field. Let $K/F$ be a finite extension of degree $r$. Furthermore, Assume $K$ 
is an imaginary and normal extension over $F$, with ring of integers $\mathcal{O}_K$.
\item[$\bullet$]
For each isomorphism class $[\phi]\in {\rm CM}(\mathcal{O}_K,\iota)$, we may choose a  representative $\phi$ is defined over $H_K$. Take the compositum field $L:=H_K\cdot K'$.  Furthermore, We fix a place $\mathfrak{P}$ of $L$ above $\mathfrak{p}$ of $F$. And consider the local field $L_{\mathfrak{P}}$ at $\mathfrak{P}$. Let $L_{\mathfrak{P}}^{nr}$ be the maximal unramified extension of $L_{\mathfrak{P}}$. We set ${\check{L}_{\mathfrak{P}}^{nr}}$ to be the completion of $L_{\mathfrak{P}}^{nr}$, and its discrete valuation ring is denoted by $W$. Let $\mu$ be a fixed uniformizer of $W$ with normalized valuation $\nu$. Therefore, up to isomorphisms, we may assume the representative $\phi$ of $[\phi]\in {\rm CM}(\mathcal{O}_K,\iota)$ is defined over $W$. 
\end{enumerate}
\end{term}

Recall from Proposition \ref{genjest}  that we have

$$\nu(J^{(\delta_1,\cdots,\delta_{r-1})}(\phi)-J^{(\delta_1,\cdots,\delta_{r-1})}(\phi'))\geqslant \frac{1}{q^r-1}\sum_{n\geqslant 1} \#{\rm Iso}_{W/\mu^nW}(\phi, \phi').$$

Hence we can directly deduce that
\begin{thm}\label{thm1'}
$$
\nu\left(\prod_{[\phi]\in{\rm CM}(\mathcal{O}_K,\iota)}J^{(\delta_1,\cdots,\delta_{r-1})}(\phi)-J^{(\delta_1,\cdots,\delta_{r-1})}(\phi')\right)\\ \geqslant \frac{1}{q^r-1}\sum_{[\phi]\in{\rm CM}(\mathcal{O}_K,\iota)}\sum_{n\geqslant 1} \#{\rm Iso}_{W/\mu^nW}(\phi, \phi').$$
\end{thm}
Going through the same proof of Theorem \ref{thm2} and Proposition \ref{algemb} respectively, while replacing $\varphi$ by $\phi'$ and use our new terminology of $W$, we have

\begin{defi}
Given $w_1, w_2 \in {\rm Iso}_{W/\mu^nW}(\phi',\phi)$, we define $w_1\sim w_2$ if and only if there is some $\xi\in\bar{\F}_{q}^*$ such that $w_2=\xi w_1$ and $$\xi^{-1}\bar{\phi}_{a}\xi=\bar{\phi}_{a} \textrm{ for } a\in \mathcal{O}_K.$$

 It is clear that $\sim$ is an equivalence relation. Hence we can divide the disjoint union $\sqcup_{[\phi]\in {\rm CM}(\mathcal{O}_K,\iota)}{\rm Iso}_{W/\mu^nW}(\phi',\phi)$ into equivalence classes with respect to $\sim$.
\end{defi}

\begin{thm}\label{thm2'}
There is a one-to-one correspondence between equivalence classes of $$\sqcup_{[\phi]\in {\rm CM}(\mathcal{O}_K,\iota)}{\rm Iso}_{W/\mu^nW}(\phi',\phi) \textrm{ with respect to }\sim,$$ and the set  
{\center$S_n:=$ $\Bigg\{\eta:\mathcal{O}_K\hookrightarrow {\rm End}_{W/\mu^nW}(\phi') \ \bigg|$
$
\begin{array}{cc}
 \eta \textrm{ is an $A$-algebra embedding},  \\
 \newline \eta|_A=\phi'\mod\mu^n, \textrm{and } \partial\circ\eta\equiv \iota \mod \mu^n  \\
\end{array}
$
$\Bigg\}.$}
\end{thm}

\begin{cor}\label{sn'}
$$\#\sqcup_{[\phi]\in {\rm CM}(\mathcal{O}_K,\iota)}{\rm Iso}_{W/\mu^nW}(\phi',\phi)=\sum_{[\phi]\in {\rm CM}(\mathcal{O}_K,\iota)}\#{\rm Iso}_{W/\mu^nW}(\phi,\phi')\geqslant (q-1)\#S_n.$$
\end{cor}

\begin{prop}\label{algemb'}
Define the set $\mathcal{M}_n$ to be
{\center$\Bigg\{\eta:\mathcal{O}_K\hookrightarrow {\rm End}_{W/\mu^nW}(\phi')\ \bigg|$
$
\begin{array}{cc}
 \eta \textrm{ is an $A$-algebra embedding }, \textrm{ and } \eta|_A=\phi'\mod\mu^n 
 \end{array}
$
$\Bigg\}.$}

Then we have
$$\#S_n\geqslant\frac{1}{\rsep}\# \mathcal{M}_n.$$

\end{prop}

Note that in the proof of Proposition \ref{algemb'}, we need that $W/F_\fp$ is a normal extension. Thus it is necessary that $K'/F$ is normal. 

Combining the Theorem \ref{thm1'}, Corollary \ref{sn'}, and Proposition \ref{algemb'} we get

\begin{thm}\label{main1'}
Let $r\geqslant 2$ be an integer, and $q=p^e$ be a prime power. Let $A=\F_q[T]$ be the polynomial ring, and $F=\F_q(T)$ be its fractional field. Let $K/F$ be an imaginary extension of degree $r$, with $\rsep=[K:F]_{\rm sep}$ defined to be its separable degree, and the integral closure of $A$ in $K$ is denoted by  $\mathcal{O}_K$. Furthermore, Assume $K$ is a normal extension over $F$. 

Fix a prime ideal $\fp$ of $A$, and a  Drinfeld module $\phi'$ of rank $r$ defined over a normal extension $K'/F$ that has good reduction at a place $\fq'$ of $K'$ stands above $\fp$.   For any basic $J$-invariant $J^{(\delta_1,\cdots,\delta_{r-1})}$, we have
$$
\begin{array}{ccc}
{\rm ord}_\fp\left(\prod_{[\phi]\in{\rm CM}(\mathcal{O}_K,\iota)}J^{(\delta_1,\cdots,\delta_{r-1})}(\phi)-J^{(\delta_1,\cdots,\delta_{r-1})}(\phi')\right)
\geqslant \frac{q-1}{\rsep\cdot(q^r-1)\cdot e_{K\cdot K',\fp}}\sum_{n\geqslant 1} \#\mathcal{M}_n \\
\end{array}.
$$
Here $e_{K\cdot K', \fp}$ is equal to the ramification index of the compositum field extension $K\cdot K'/F$ at the place $\fp$, and $ \mathcal{M}_n$ is the set of $A$-algebra embeddings $\eta:\mathcal{O}_K\hookrightarrow{\rm End}_{W/\mu^nW}(\phi')$ with $\eta|_A=\phi'\mod \mu^n$.

\end{thm}

\section{Matrix realization}
In this section, we make further assumption that
\begin{enumerate}

 \item[$\bullet$] $\varphi_T=T+\tau^r$ has good supersingular reduction at $\fp$. By Proposition \ref{ssred} (ii), it is equivalent to say that the monic generator $\pi$ of $\fp$ has degree $\deg_T(\pi)$ coprime to $r$. 
  
\end{enumerate}
Moreover, for the purpose of computation, we may assume the following:
\begin{enumerate}
\item[$\bullet$] $K=F(s_1,\cdots, s_t)$ is a degree-$r$ imaginary and normal extension over $F$, where $\mathcal{O}_K=A[s_1,\cdots, s_t]$ and $t$ minimal. We can write $A[s_1,\cdots,s_t]\cong A[X_1,\cdots,X_t]/I_{\mathcal{O}_K}$ for some ideal $I_{\mathcal{O}_K}$ in $A[X_1,\cdots, X_t]$.

\end{enumerate}

Theorem \ref{thm2} says counting number of isomorphism classes is equivalent to count the number of elements in $S_{n}$. A matrix realization for ${\rm End}_{W/\mu^n W}(\varphi)$ will translate the counting problem into counting certain type of integral matrices.  

At the beginning, we recall some well-known results on endomorphism algebra of the supersingular Drinfeld module $\underline{\varphi}:=\varphi \mod \fp$. Here $\varphi$ is the rank-$r$ Drinfeld $A$-module $\varphi_T=T+\tau^r$.

\begin{prop}\label{alg}
Let $D:={\rm End}^{\circ}(\underline{\varphi})$, we have
\begin{enumerate}
\item[(a)] $D$ is a central division algebra over $F$ with $[D:F]=r^2$.
\item[(b)] $H=\F_{q^r}(T)$ embeds into $D$ as an $F$-subalgebra.

\item[(c)] The Hasse invariant of $D$ is

$${\rm Inv}_\nu (D)=\Bigg\{
\begin{array}{cc}
\frac{1}{r},&\ \nu=\fp \\
\frac{-1}{r},& \ \nu=\infty \\
0,& \textrm{ otherwise}

\end{array}.$$

\end{enumerate}
\end{prop}

As an application of the proposition \ref{alg} (b), we may view ${\rm End}_{W/\mu W}(\varphi)\otimes_A \mathcal{O}_H$ as an $\mathcal{O}_H$-algebra. There is an embedding of ${\rm End}_{W/\mu W}(\varphi)$ into a matrix algebra: 
$${\rm End}_{W/\mu W}(\varphi)\otimes
_A \mathcal{O}_H={\rm End}(\underline{\varphi})\otimes_A \mathcal{O}_H \hookrightarrow D\otimes_F H\simeq M_r(H).$$

Furthermore, let $\sigma$ be a generator of ${\rm Gal}(H/F)$ . Hasse's main theorem on algebra, Gruwald-Wang theorem (see \cite{BG16}, (1.2)) and proposition \ref{alg} (c) implies that the image of $D\otimes_F H$ is isomorphic in $M_r(H)$ to the cyclic algebra 
$$(H/F,\sigma,\pi ):=H[\tau,\tau^{-1}]/(\tau^r-\pi)H[\tau,\tau^{-1}].$$
Here the multiplication in $L[\tau,\tau^{-1}]$ is defined by
$$\alpha \tau^n \cdot \beta \tau^m=\alpha\sigma^n(\beta)\tau^{n+m}.$$

The image of $(H/F,\sigma,\pi)$ in $M_r(H)$ is clear. Viewing $(H/F,\sigma,\pi)$ as an $r$-dimensional $H$-vector space with basis $\mathcal{B}=\{\tau^i, 0\leqslant i\leqslant r-1\}$, left multiplication with respect to $\mathcal{B}$ induces the following map:

$$
\begin{array}{cccc}
\eta:& H^r& \longrightarrow& M_r(H)\\

       & (x_1,\cdots,x_r)&\mapsto&M_{(x_1,\cdots,x_r)}:= \left(\begin{array}{cccc}x_1 & x_2 & \cdots & x_r \\\pi\sigma(x_r) & \sigma(x_1) & \cdots & \sigma(x_{r-1}) \\\vdots &  & \ddots & \vdots \\ \pi\sigma^{r-1}(x_2) & \cdots & \pi\sigma^{r-1}(x_r) & \sigma^{r-1}(x_1)\end{array}\right)

\end{array}
$$

Moreover, $\mathcal{O}_H[\tau,\tau^{-1}]/(\tau^r-\pi)\mathcal{O}_H[\tau,\tau^{-1}]$ is a maximal order of $(H/F,\sigma,\pi)$. This is described in \cite{BG16}, section 4.5. Note that our extension $H/F$ is a constant extension.

On the other hand, ${\rm End}(\underline{\varphi})$ is a maximal order in $D$. Hence  ${\rm End}(\underline{\varphi})\otimes_A \mathcal{O}_H$ is a maximal order of $D\otimes_F H$.
 Similar to the rank-$2$ case, since the class number of $H$ is $1$ (i.e. $\mathcal{O}_H$ is a P.I.D.), there is only one maximal order in $D\otimes_F H$ where $\mathcal{O}_H$ embeds optimally.

 Thus we have the matrix realization of ${\rm End}_{W/\mu W}(\varphi)={\rm End}(\underline{\varphi})$ in $M_r(H)$. In conclusion, we get the following proposition:
\begin{prop}\label{endomu}
Under the embedding $${\rm End}_{W/\mu W}(\varphi)\otimes
_A \mathcal{O}_H \hookrightarrow D\otimes_F H\simeq M_r(H).$$ We have, up to conjugation in $M_r(H)$, that
$$
{\rm End}_{W/\mu W}(\varphi)=\left\{   \left(\begin{array}{cccc}x_1 & x_2 & \cdots & x_r \\\pi\sigma(x_r) & \sigma(x_1) & \cdots & \sigma(x_{r-1}) \\\vdots &  & \ddots & \vdots \\ \pi\sigma^{r-1}(x_2) & \cdots & \pi\sigma^{r-1}(x_r) & \sigma^{r-1}(x_1)\end{array}\right)
             \ \bigg| \   x_i\in\mathcal{O}_H    \textrm{ for all }i             \right\}:=\mathcal{M}. 
$$
\end{prop}
Now we prove a matrix realization for ${\rm End}_{W/\mu^n W}(\varphi)$ with $n\geqslant 2$.

\begin{prop}\label{punr}
Suppose that $\fp$ is unramified in $K/F$, then we have the endomorphism ring

$${\rm End}_{W/\mu^n W}(\varphi)=\mathcal{O}_H+\pi^{n-1}{\rm End}_{W/\mu W}({\varphi})$$

\end{prop}

The strategy of this proof mainly follows section 3 in \cite{G86}. In order to prove Proposition \ref{punr}, we prove a key lemma first.  Let $\mathcal{O}$ be the ring of integers in $H_\fp$, the completion of $H$ at $\fp$. On the other hand, let $M_\fp$ be the completion of the maximal unramified extension of $H_\fp$, and $W'$ be the ring of integers in $M_\fp$. Then we have $W'=W$ and we may choose a uniformizer $\mu$ of $W$ to be $\pi$.

We consider the formal $A_\fp$-module $\underline{G}=(\underline{G}(x,y):=x+y,\  \underline{\varphi}(x))$, where $\underline{\varphi}(x)$ is induced from the Drinfeld module $\underline{\varphi}:A\rightarrow k\{\tau\}$ (see \cite{G86} (1.1) \& (1.2) for definition of formal $A_\fp$-module). Since $\varphi$ has supersingular reduction at $\fp$, we can see that $\underline{G}$ is a formal $A_\fp$-module of height $r$. 

 Denote $R:={\rm End}_{k}(\underline{G})$. From the fact that $\mathcal{O}$ is the completion of $H={\rm End}_W(\varphi)$ at $\fp$, and ${\rm End}_W(\varphi)\hookrightarrow {\rm End}_k(\underline{\varphi})$, we have an embedding $$\mathcal{O}\hookrightarrow R.$$ We may fix an embedding $\iota:\mathcal{O}\hookrightarrow R$ such that the action on ${\rm Lie}_k(\underline{G})$ is the reduction map modulo $\mu$.  Thus we may view $\underline{G}$ as a formal $\mathcal{O}$-module over $k'$ of height $1$. 

Similar to proposition 2.1 in \cite{G86}, there is a formal $\mathcal{O}$-module $G=(G(x,y), g(x))$ over $W$ (the canonical lifting of the pair $(\underline{G},\iota)$),  with $g_{\pi}(x)=\pi x+x^{q^{r\deg_T \pi}}$. The formal $\mathcal{O}$-module $G$ reduces to $\underline{G}$ after reduction modulo $\pi$. Such a  $\mathcal{O}$-module $G$ over $W$ is unique up to $W$-isomorphism. Note that in our proof the objects  ``$\underline{G}$ over $k$'' and ``$G$ over $W$'' are reversed in Gross' proof, where the objects are formal $\mathcal{O}$-modules ``$G$ over $k$'' and `` $\underline{G}$ over $W$ ''.

For $n\geqslant 1$, define $R_0=R$ and $R_{n-1}={\rm End}_{W/\pi^n W}(G)$. The reduction map modulo $\pi^n$ induces the injections
$$R_n\hookrightarrow R_{n-1}\hookrightarrow \cdots \hookrightarrow R_1\hookrightarrow R_0=R.$$
From the construction of $G$, we can see that ${\rm End}_{W}(G)=\mathcal{O}$. Also, we know that $\mathcal{O} \subseteq R_n$ for all $n\geqslant 1$.

\begin{lem}\label{keylem}

As $\mathcal{O}$-modules, we have $$ R_n=\mathcal{O}+\pi^n R \textrm{ for all } n\geqslant1.$$

\end{lem}

\begin{proof}

The main part of the proof of claim is the same as in proposition 3.3 \cite{G86}, the only difference is that $R_n$ in our situation is a free $\mathcal{O}$-module of rank $r$ for $n\geqslant 0$.

First of all, we may assume $G=(G(x,y),g(x))$ is given by a formal group law $G(x,y)\in W'[[x,y]]$. Then we apply Drinfeld's formal cohomology theory (see \cite{D74} section 4, pp 570-571). Let $\underline{G}_1(x,y)$ be the partial derivative of $\underline{G}(x,y)$ with respect to $x$. Then we have $\underline{G}_1=1$. For an endomorphism $f(x)\in R_{n-1}$, we compute the following series:
$$\alpha_f(x,y):=G_1(0,G(f(x),f(y)))^{-1}=1 \textrm{ in } k'[[x,y]], \textrm{ and } \beta_f(x):=G_1(0,f(x))^{-1}=1  \textrm{ in } k'[[x]].$$

The data 
\begin{align*}{}
 \triangle(f)(x,y):=\alpha_f(x,y)[f(G(x,y))-G(f(x),f(y))]\\
\delta_a(f)(x):=\beta_f(x)[f(a(x))-a(f(x))] \ \textrm{ for }a\in A
\end{align*}

defines a symmetric $2$-cocycle of the $G$ with coefficients in $\pi^nW/ \pi^{n+1} W$. The cohomology class $(\triangle, \delta_a)$ depends only on the class of $f\in R_{n-1}/ R_n$. Thus we have the following map

$$\alpha_n: R_{n-1}/R_n\hookrightarrow H^2(G,\pi^n W/\pi^{n+1} W).$$

Next, we compute $\triangle(g_\pi\circ f)(x,y)$ and $\delta_a(g_\pi \circ f)$. One can see that the $\mathcal{O}$-module $R_{n-1}/R_n$ is annihilated by $\pi$. We further have the following commutative diagram:
$$
\begin{array}{cccc}
\alpha_n:& R_{n-1}/R_n&\longrightarrow&H^2(G,\pi^n W'/\pi^{n+1} W')\\
\ &\hookdownarrow{\pi} &\                 &{\simeq}\downarrow{\pi}\\
\alpha_{n+1}:& R_n/R_{n+1}& \longrightarrow          &H^2(G,\pi^{n+1} W'/\pi^{n+2} W')
\end{array}.
$$
Thus to prove the claim, it is enough to prove the following two statements:

\begin{enumerate}
\item[(a) ]$R=\mathcal{O}+\mathcal{O}\tau+\cdots+\mathcal{O}\tau^{r-1},\textrm{ where } \tau(x):=x^q.$\item[(b) ]$R/R_1\cong (\mathcal{O}/\pi \mathcal{O})^{r-1}$ as $\mathcal{O}$-modules.

\end{enumerate}

The statement (a) is clear since we have
$${\rm End}_{k}(G)=R\overset{\mathrm{def}}{=}{\rm End}_{k}(\underline{G}).$$
Recall that $\underline{G}$ is a formal $A_\fp$-module induced from the Drinfeld $A$-module $\underline{\varphi}_T=T+\tau^r$ over $k=W/\pi W$. As ${\rm End}_{k}(\underline{G} |_A)={\rm End}_{k}(\underline{\varphi})$ is a free $\mathcal{O}_H$-module of rank $r$, and completion at $\fp$ does not increase the rank of  ${\rm End}_{k}(\underline{G})$ as $\mathcal{O}$-module. Moreover, one can check that $\tau(x):=x^q$ lies in ${\rm End}_{k}(\underline{G})$ by using the definition of endomorphism ring for formal $\mathcal{O}$-module  (see (1.3) in \cite{G86} ). Thus we have
$R\supset \mathcal{O}+\mathcal{O}\tau+\cdots+\mathcal{O}\tau^{r-1}.$ Hence we get the equality from counting the rank on both sides. One can also follow Gross' computation for the case ``$\mathcal{O}$ unramified over $A$'' in arguments after (3.6) in \cite{G86}, page 323.

For the statement (b), we know that $R/R_1$ is annihilated by $\pi$, statement (a): $R=\mathcal{O}+\mathcal{O}\tau+\cdots+\mathcal{O}\tau^{r-1}$, and $\mathcal{O}\subset R_1$. The fundamental theorem for module over PID shows that 
$$R_1\supseteq\mathcal{O}+\pi\mathcal{O}\tau+\cdots+\pi\mathcal{O}\tau^{r-1}.$$

Suppose that the containment is proper, then there must be some $\tau^i$ lies in $R_1$ with $1\leqslant i\leqslant r-1$. Then we check by definition that $\tau^i\not\in {\rm End}_{W/\pi^2 W}(G)$. This gives a contradiction. Hence the proof of (b) is complete.

\end{proof}

\begin{proof}[proof of Proposition \ref{punr}]
 
With Lemma \ref{keylem} in hand, the proof is now just an application of $${\rm End}_{W/\pi^{n+1} W}(G)=\mathcal{O}+\pi^n {\rm End}_{W/\pi W}(G) \textrm{ for all } n\geqslant1.$$
 by taking restriction of the formal $\mathcal{O}$-module $G$ at $\mathcal{O}_H$ .

\end{proof}

Combining the above proposition with our matrix realization of ${\rm End}_{W/\mu W}(\varphi)$( see Proposition \ref{endomu}), we have following interpretation in the case $\fp$ is unramified in $K/F$.
\begin{equation}
{\rm End}_{W/\mu^n W}(\varphi)=\left\{M_{(x_1,\cdots,x_r)} \in \mathcal{M}\mid x_i\equiv 0 \mod \pi^{n-1}    \textrm{ for } 2\leqslant i\leqslant r    \right\}. \tag{5}
\end{equation}

The remaining case is that $\fp$ ramifies in $K/F$. Since  $K/F$ is normal, its ramification index at $\fp$ is independent of the choice of primes of $K$ stand above $\fp$. Thus we denote by $e_{K,\fp}$ the ramification index of $\fp$ in $K$.

\begin{prop}\label{pram}
Suppose that $\fp$ is ramified in $K/F$, we have
\begin{equation}
{\rm End}_{W/\mu^n W}(\varphi)=\left\{M_{(x_1,\cdots,x_r)} \in \mathcal{M} \mid x_i\equiv 0 \mod \pi^{z-1}    \textrm{ for } 2\leqslant i\leqslant r    \right\}, \tag{6}
\end{equation}
where $z:=\lfloor\frac{n+e_{K,\fp}-1}{e_{K,\fp}}\rfloor$.
\end{prop}
\begin{proof}

The arguments are similar to \cite{D91} pp. 248-249.  Let $\mu$, $W$ and $\pi$, $W'$ be as defined before. We have $\mu^{e(\fq/\fp)}=\pi\cdot u$ for some unit $u$ in $W'$. Firstly, follow the arguments in the  ``unramified'' case while replacing $W$ by $W'$ and $\mu$ by $\pi$, we still have
$${\rm End}_{W'/\pi^n W'}(\varphi)=\left\{M_{(x_1,\cdots,x_r)} \in \mathcal{M}\mid x_i\equiv 0 \mod \pi^{n-1}    \textrm{ for } 2\leqslant i\leqslant r    \right\}.$$

Now we look into ${\rm End}_{W/\pi^n W}(\varphi)$. Since $\varphi_T=T+\tau^r$ has no new endomorphisms over $W$, the endomorphism rings ${\rm End}_{W/\pi^nW}(\varphi)={\rm End}_{W'/\pi^nW'}(\varphi)$ contain all the endomorphisms of $\varphi$ modulo power of $\mu$. Thus we have
$${\rm End}_{W/\mu^n W}(\varphi)={\rm End}_{W'/(\mu^n W\cap W')}(\varphi).$$  

From the equality $\mu^{e(\fq/\fp)}=\pi\cdot u$, we have the following two identities:
$${\rm End}_{W/\mu^{e(\fq/\fp)\cdot n-k} W}(\varphi)={\rm End}_{W/\mu^{e(\fq/\fp)\cdot n} W}(\varphi) \textrm{ for } 1\leqslant k\leqslant e(\fq/\fp)-1; $$$${\rm End}_{W/\mu^{e(\fq/\fp)\cdot n} W}(\varphi)={\rm End}_{W'/\pi^n W'}(\varphi).$$
Hence the proposition follows from equalities above and the structure of ${\rm End}_{W'/\pi^n W'}(\varphi)$.
\end{proof}

On the other hand, we give a computational characterization of the sets $S_n$ and $\mathcal{M}_n$.

\begin{lem}\label{snmn}
We have {\center$S_n=$ $\Bigg\{(\alpha_1,\cdots, \alpha_t)\ \bigg|$
$
\begin{array}{cc}
 \alpha_i\in {\rm End}_{W/\mu^nW}(\varphi) \textrm{ such that } \alpha_i\alpha_j=\alpha_j\alpha_i \textrm{ for }1\leqslant i, j\leqslant t,  \\ f(\alpha_1,\cdots,\alpha_t)=0 \textrm{ for all } f\in I_{\mathcal{O}_K},  \\
 \newline \textrm{ and }  \partial(\alpha_i)\equiv \iota(s_i) \mod \mu^n  \textrm{ for all }1\leqslant i\leqslant t   &  \\
\end{array}
$
$\Bigg\},$} and  

{\center$\mathcal{M}_n=$ $\Bigg\{(\alpha_1,\cdots, \alpha_t)\ \bigg|$
$
\begin{array}{cc}
 \alpha_i\in {\rm End}_{W/\mu^nW}(\varphi) \textrm{ such that } \alpha_i\alpha_j=\alpha_j\alpha_i \textrm{ for }1\leqslant i, j\leqslant t,  \\ f(\alpha_1,\cdots,\alpha_t)=0 \textrm{ for all } f\in I_{\mathcal{O}_K}
\end{array}
$
$\Bigg\},$} 

\end{lem}
\begin{proof}

The characterization of $\M_n$ is immediate once that of $S_n$ is done. Now consider an $A$-algebra embedding $\eta: \mathcal{O}_K\rightarrow {\rm End}_{W/\mu^nW}(\varphi)$ in $S_n$. We assign $\alpha_i$ to be the image $\eta(s_i)$ of $s_i$ in ${\rm End}_{W/\mu^nW}(\varphi)$. Since $\eta$ is an $A$-algebra homomorphism, we have $\alpha_i\alpha_j=\alpha_j\alpha_i$ for all $1\leqslant i,j\leqslant t$. Now for $f(X_1,\cdots,X_t)\in I_{\mathcal{O}_K}$, we have
$$f(\alpha_1,\cdots, \alpha_t)=\eta(f(s_1,\cdots,s_t))=\eta(0)=0.$$
And the last condition holds because $\partial(\alpha_i)=\partial\circ \eta(s_i)\equiv \iota(s_i)\mod \mu^n$.

Conversely, for $(\alpha_1,\cdots,\alpha_t)$ satisfies the three conditions, we immediately have an $A$-algebra homomorphism $$\eta: \mathcal{O}_K\rightarrow {\rm End}_{W/\mu^nW}(\varphi)$$ generated by $$\eta(a)=\varphi_a \mod \mu^n \textrm{ for } a\in A,\textrm{ and }\eta(s_i)=\alpha_i \textrm{ for }1\leqslant i\leqslant t.$$ 
In fact, the homomorphism $\eta$ is an embedding. To prove this, we suppose on contrary that ${\rm ker}\eta $ is a nontrivial ideal of $\mathcal{O}_K$. Then ${\rm ker}\eta\cap A$ is a nontrivial ideal of $A$. Therefore, there is some $a\in A$ such that $\varphi_a\equiv 0 \mod \mu^n$, which is a contradiction. 

Therefore, we get a one-to-one correspondence between $S_n$ and the collection of $(\alpha_1,\cdots,\alpha_t)$ satisfying the three conditions.
\end{proof}
As a result, we have

\begin{cor}\label{estj}
Under the same assumptions as in Theorem \ref{main1}, with an additional restriction that 
\begin{enumerate}

\item[$\bullet$] $\fp=(\pi)$ is a prime ideal of $\F_q[T]$ whose degree is coprime to $r$.

\end{enumerate}

Then

\begin{align*}
{\rm{ord}}_\fp(J_{\mathcal{O}_K}^{(\delta_1,\cdots,\delta_{r-1})})&\geqslant \frac{(\sum_{i=1}^{r-1}\delta_i)(q-1)}{r_{\rm sep}\cdot(q^r-1)\cdot e_{K,\fp}}\sum_{m\geqslant 0} \#\mathcal{M}_{m\cdot e_{K,\fp}+1}. \\
\end{align*}
Here $
\#\M_{m\cdot e_{K,\fp}+1}=$ 

$\#\Bigg\{(\alpha_1,\cdots, \alpha_t)\in \mathcal{M}^t\ \bigg|$
$
\begin{array}{cc}
 \alpha_i(x_1,\cdots,x_r)\in\mathcal{M} \textrm{ satisfies } x_k\equiv 0 \mod \pi^{m} \textrm{ for } 2\leqslant k\leqslant r, \\ \alpha_i\alpha_j=\alpha_j\alpha_i \textrm{ for }1\leqslant i, j\leqslant t,  \\ f(\alpha_1,\cdots,\alpha_t)=0 \textrm{ for all } f\in I_{\mathcal{O}_K}  \\
 \end{array}
$
$\Bigg\}$
, and $
\mathcal{M}=\left\{   \left(\begin{array}{cccc}x_1 & x_2 & \cdots & x_r \\\pi\sigma(x_r) & \sigma(x_1) & \cdots & \sigma(x_{r-1}) \\\vdots &  & \ddots & \vdots \\ \pi\sigma^{r-1}(x_2) & \cdots & \pi\sigma^{r-1}(x_r) & \sigma^{r-1}(x_1)\end{array}\right)
             \ \bigg| \   x_i\in \F_{q^r}[T]    \textrm{ for all }i \right\}.           
$

\end{cor}
\begin{proof}
If $\fp$ unramified in $K/F$, then the result comes directly from Theorem \ref{thm1}, Corollary \ref{sn}, and equality (5).

If $\fp$ ramifies in $K/F$. Then from equality (6) we have $${\rm End}_{W/\mu^n W}(\varphi)=\left\{M_{(x_1,\cdots,x_r)} \in \mathcal{M} \mid x_i\equiv 0 \mod \pi^{z-1}    \textrm{ for } 2\leqslant i\leqslant r    \right\},$$ $\textrm{ where } z:=\lfloor\frac{n+e_{K,\fp}-1}{e_{K,\fp}}\rfloor$. Therefore,  for $m\cdot e_{K,\fp}+1\leqslant n\leqslant (m+1)\cdot e_{K,\fp}$, 
$${\rm End}_{W/\mu^n W}(\varphi)={\rm End}_{W/\mu^{m\cdot e(\fq/\fp)+1} W}(\varphi).$$
This means that for an $\alpha\in {\rm End}_{W/\mu^n W}(\varphi)$, the constant term $\partial\circ\alpha$ is actually an element in $W/\mu^{m+1}W$. Now from Lemma \ref{snmn} we have
{\center$S_n=$ $\Bigg\{(\alpha_1,\cdots, \alpha_t)\ \bigg|$
$
\begin{array}{cc}
 \alpha_i\in {\rm End}_{W/\mu^nW}(\varphi) \textrm{ such that } \alpha_i\alpha_j=\alpha_j\alpha_i \textrm{ for }1\leqslant i, j\leqslant t,  \\ f(\alpha_1,\cdots,\alpha_t)=0 \textrm{ for all } f\in I_{\mathcal{O}_K},  \\
 \newline \textrm{ and }  \partial(\alpha_i)\equiv \iota(s_i) \mod \mu^n  \textrm{ for all }1\leqslant i\leqslant t   &  \\
\end{array}
$
$\Bigg\}.$}

Combining with Theorem \ref{thm1}, we get

$$
\begin{array}{ccc}
{\rm{ord}}_\fp(J_{\mathcal{O}_K}^{(\delta_1,\cdots,\delta_{r-1})})&\geqslant& \frac{(\sum_{i=1}^{r-1}\delta_i)(q-1)}{r_{\rm sep}\cdot(q^r-1)\cdot e_{K,\fp}}\sum_{n\geqslant 1} \#\mathcal{M}_n \\
\ \\
&\geqslant& \frac{(\sum_{i=1}^{r-1}\delta_i)(q-1)}{r_{\rm sep}\cdot(q^r-1)\cdot e_{K,\fp}}\sum_{m\geqslant 0} \#\mathcal{M}_{m\cdot e(\fq/\fp)+1}
\end{array}.
$$
Here $\mathcal{M}_{m\cdot e_{K,\fp}+1}$ is the collection of $(\alpha_1,\cdots, \alpha_t) \in \mathcal{M}^t$ such that $$\alpha_i(x_1,\cdots,x_r)\in {\rm End}_{W/\mu^{m\cdot e_{K,\fp}+1} W}(\varphi)=\left\{M_{(x_1,\cdots,x_r)} \in \mathcal{M} \mid x_i\equiv 0 \mod \pi^{m}    \textrm{ for } 2\leqslant i\leqslant r    \right\}$$
satisfy the conditions ``$\alpha_i$'s are mutually commutative with each other '' and ``$f(\alpha_1,\cdots,\alpha_t)=0$ for all $f\in I_{\mathcal{O}_K}$''. This completes the proof.

\end{proof}

\begin{rem}\label{estone}
When there is only one generator $s$ in the degree-$r$ imaginary extension $K/F$, and $\mathcal{O}_K=A[s]$, we can further simplify Corollary \ref{estj} into the following:
$$
{\rm{ord}}_\fp(J_{\mathcal{O}_K}^{(\delta_1,\cdots,\delta_{r-1})})\geqslant \frac{(\sum_{i=1}^{r-1}\delta_i)(q-1)}{r_{\rm sep}\cdot(q^r-1)\cdot e_{K,\fp}}\sum_{m\geqslant 0} \#\mathcal{M}_{m\cdot e_{K,\fp}+1} $$
, where
\begin{align*}
\#\mathcal{M}_{m\cdot e_{K,\fp}+1}=\#\Big\{\alpha_{(x_{1},\cdots, x_{r})} \in \mathcal{M}\mid x_i\equiv 0 \mod \pi^{m}    \textrm{ for } 2\leqslant i\leqslant r,\\  \textrm{ min. poly of } \alpha=\textrm{ min. poly of }s     \Big\}.
\end{align*}

The reason is because when $K=F(s)$ and the integral closure of $A$ in $K$ is $\mathcal{O}_K=A[s]$, we actually have
$$\mathcal{O}_K=A[s]=A[X]/(m(X)).$$
Here $m(X)\in A[X]$ is the minimal polynomial of $s$. Therefore, for an element $\alpha\in \mathcal{M}_n$, the condition ``$f(\alpha)=0$ for all $f\in I_{\mathcal{O}_K}$'' reduces to ``$m(\alpha)=0$''. Moreover, because $m(\alpha)=0$ we know that $\partial\circ \alpha$ is congruent to some conjugate of $s$ modulo $\mu^n$. This forces $m(X)$ to be not only the minimal polynomial of $s$, but also the minimal polynomial of $\alpha$ in the $A$-subalgebra $A[\alpha]$ of ${\rm End}_{W/\mu^nW}(\varphi)$. 

\end{rem}

\section{Examples}

In this section, we apply our estimation of  singular moduli to cubic imaginary extension case. For simplicity, let $q=p^e$ be an odd prime power such that $\F_q$ contains the cubic roots of unity, e.g. $q=7$. Let $K:=F(\sqrt[3]{\Delta})$ over $F=\F_q(T)$, where $\Delta\in A=\F_q[T]$ is cubic-free with $T$-degree coprime to $3$. Then $K/F$ is a Kummer extension, hence $K/F$ is normal. Suppose that $\mathcal{O}_K=A[\sqrt[3]{\Delta}]$.

 On the other hand, we know that at here $H=\F_{q^3}(T)$ is the Hilbert class field of $\varphi_T=T+\tau^3$, and $\mathcal{O}_H=\F_{q^3}[T]$. Let $\sigma: \alpha \mapsto \alpha^q$ be a generator of ${\rm Gal}(H/F)$. And let $\fp=(\pi)$ be a prime ideal of $A$ with $\deg_T(\pi)$ coprime to $3$, then $\fp$ is a supersingular reduction prime of $\varphi$.

From Remark \ref{estone}, we have
$$
\begin{array}{ccc}
{\rm{ord}}_\fp(J_{\mathcal{O}_K}^{(\delta_1,\delta_2)})&\geqslant \frac{(\delta_1+\delta_2)(q-1)}{r_{\rm sep}\cdot (q^3-1)\cdot e_{K,\fp}} \sum_{m\geqslant 0} \#\mathcal{M}_{m\cdot e_{K,\fp}+1} . 

\end{array}
$$
And we know that the counting number of elements in $\M_n$ for $n\geqslant 1$ an integer is equal to count the number of matrices

$$\left(\begin{array}{ccc}x_1 & x_2  & x_3 \\ \pi\sigma(x_3) & \sigma(x_1) & \sigma(x_2) \\\pi\sigma^{2}(x_2) & \pi\sigma^{2}(x_3) & \sigma^{2}(x_1)\end{array}\right)
,$$
whose minimal polynomial is equal to $X^3-\Delta$. Here $x_1\in \mathcal{O}_H$, and $ x_i=\pi^{n-1}x_i' $ with $x_i'\in\mathcal{O}_H$ for $2\leqslant i\leqslant 3$. Note that $\pi\in A[T]$ is the monic generator of $\fp$, so $\pi$ is invariant under the action of $\sigma$. Comparing the coefficients of the minimal polynomial of such matrices with the minimal polynomial $X^3-\Delta$ of $\sqrt[3]{\Delta}$, we get three equations

$$ {\rm Tr}(x_1)=0, $$
$$x_1\sigma(x_1)+x_1\sigma^2(x_1)+\sigma(x_1)\sigma^2(x_1) = \pi^{2n+1}[\sigma^2(x_2')x_3'+\sigma(x_2')\sigma^2(x_3')+x_2'\sigma(x_3')], $$ 
and
$$\begin{array}{lll}\Delta&=&{\rm Norm}(x_1)+\pi^{3n-2}{\rm Norm}(x_2')+\pi^{3n-1}{\rm Norm}(x_3')\\&&-\pi^{2n-1}[\sigma(x_1)\sigma^2(x_2')x_3'+x_1\sigma(x_2')\sigma^2(x_3')+\sigma^2(x_1)x_2'\sigma(x_3')]\end{array}.$$

After some simplification, we get

\[
\left\{
\begin{array}{ll}
 {\rm Tr}(x_1)=0 & \refstepcounter{equation}(7)\\
 \ &\ \\
 {\rm Tr}(x_1\sigma(x_1))=-\frac{1}{2}{\rm Tr}(x_1^2) = \pi^{2n+1}{\rm Tr(x_2'\sigma(x_3'))} & \refstepcounter{equation}(8)\\
 \ &\ \\
 
 {\rm Norm}(x_1)+\pi^{3n-2}{\rm Norm}(x_2')+\pi^{3n-1}{\rm Norm}(x_3')-\pi^{2n-1}{\rm Tr}(x_1\sigma(x_2')\sigma^2(x_3'))=\Delta& \refstepcounter{equation}(9)

\end{array}
\right.
.\]

Hence $\# \M_n= \#\{(x_1,x_2',x_3')\in \F_{q^3}[T]^3\mid (x_1,x_2',x_3') \textrm{ satisfies equation (4) to (6)} \}$

\begin{rem}
One can compare the equations (7) to (9) with the Diophantine equation in \cite{D91}, Lemma 5.8. Since we don't have a nice discriminant equation comparing to the rank-$2$ case, we don't know whether the equations (7) to (9) can be further simplified into a system of Diophantine equations.

\end{rem}

\subsection{The case  \texorpdfstring $\ \Delta=\pi=T$}
\subsubsection{Separable case}

We further restrict $q$ to be an odd prime power that is coprime to $3$. Before we start the computation, one can check directly that $\mathcal{O}_K=A[\sqrt[3]{T}]$. Moreover, $\mathcal{O}_K$ is a principal ideal domain, hence there is only one isomorphism class $[\phi]$ in ${\rm{CM}}(\mathcal{O}_K,\iota)$.

For simplicity, let's set $\theta=\sqrt[3]{T}$. Let $\phi'_\theta=\theta+\tau$ be the rank-$1$ Drinfeld $\mathcal{O}_K-$module. From $\phi'$, we can construct a Drinfeld $A-$module $\phi$ which has CM by $\mathcal{O}_K$.
$$
\begin{array}{lll}
\phi_T&=&\phi'_\theta\cdot\phi'_\theta\cdot\phi'_\theta\\
&=&T+[\theta^2+\theta^{q+1}+\theta^{2q}]\tau+[\theta+\theta^q+\theta^{q^2}]\tau^2+\tau^3
\end{array}.$$

Hence ${\rm{CM}}(\mathcal{O}_K,\iota)=\{[\phi]\}.$

Now we show under a suitable choice of $J^{(\delta_1,\delta_2)}$, the equality in the bound can be reached. Take $$J^{(\delta_1,\delta_2)}=J^{(0,q^2+q+1)}=\frac{g_2^{q^2+q+1}}{g_3^{q+1}}.$$

We have ${\rm{ord}}_\fp(J_{\mathcal{O}_K}^{(0,q^2+q+1)})=\frac{q^2+q+1}{3}$ by direct computation. On the other hand, by counting $T$-degree on both sides of equation (12),  one can see that such a solution $(x_1,x_2',x_3')$ exists only when $n=1$. Moreover, one can pick triples $(0,x_2',0)$ where $x_2'\in\F_{q^3}[T]$ has norm equal to $1$. These triples satisfy equation (7) to (9) when $n=1$, and there are $q^2+q+1$ many such triples. Thus we have
$$\sum_{m\geqslant 0} \#\mathcal{M}_{m\cdot e_{K,\fp}+1}=\#\M_1\geqslant q^2+q+1.$$
Therefore, 
$$\frac{q^2+q+1}{3}={\rm{ord}}_\fp(J_{\mathcal{O}_K}^{(0,q^2+q+1)})\geqslant\frac{(\delta_1+\delta_2)(q-1)}{3\cdot (q^3-1)\cdot 3}\sum_{n\geqslant 1} \#\M_n=\frac{1}{9}\cdot \#\M_1\geqslant\frac{q^2+q+1}{9}.$$  

The difference between both sides is because of the term ``$r_{\rm sep}$'' in Remark \ref{estone}. From Remark \ref{snequal} and the fact that $x^3-T\equiv x^3 \mod \sqrt[3]{T}$ has only one root,  we actually have $$\#S_1=\frac{\#\M_1}{\# \textrm{ of disjoint subsets }\{f\in\mathcal{M}_n\mid \partial\circ f=\zeta\mod \sqrt[3]{T}\} \textrm{ in }\M_1}=\frac{\#\M_1}{1},$$
while our lower bound in Remark \ref{estone} takes the smallest possible number $$\# S_1\geqslant\frac{1}{r_{\rm sep}}\#\M_1=\frac{\#\M_1}{3}.$$

\subsubsection{Inseparable case}

We consider the above example in the case when $q$ is a $3$-power, i.e. $F=\F_q(T)$ and $K=F(\sqrt[3]{T})$. It is easy to see that $K/F$ is a purely inseparable, normal, and imaginary extension over $F$. Set $\theta=\sqrt[3]{T}$. Again, because $\mathcal{O}_K=A[\theta]$ is a principal ideal domain, there is only one isomorphism class in ${\rm CM}(\mathcal{O}_K,\iota)$. The Drinfeld $A-$module $\phi$ defined by
$$
\begin{array}{lll}
\phi_T&=&T+[\theta^2+\theta^{q+1}+\theta^{2q}]\tau+[\theta+\theta^q+\theta^{q^2}]\tau^2+\tau^3
\end{array}.$$
which has CM by $\mathcal{O}_K$. Hence ${\rm{CM}}(\mathcal{O}_K,\iota)=\{[\phi]\}.$ 

Now we know that ${\rm{ord}}_\fp(J_{\mathcal{O}_K}^{(0,q^2+q+1)})=\frac{q^2+q+1}{3}$ by direct computation. From the same argument as in the separable case, we have
$$\sum_{m\geqslant 0} \#\mathcal{M}_{m\cdot e_{K,\fp}+1}=\#\M_1\geqslant q^2+q+1.$$
Thus we can compute from Corollary \ref{estj} and get

$$\frac{q^2+q+1}{3}={\rm{ord}}_\fp(J_{\mathcal{O}_K}^{(0,q^2+q+1)})\geqslant\frac{(\delta_1+\delta_2)(q-1)}{1\cdot (q^3-1)\cdot 3}\sum_{n\geqslant 1} \#\M_n\geqslant\frac{q^2+q+1}{3}.$$  
The equality is reached in this case. Thus we actually have
$$\sum_{n\geqslant 1} \#\M_n=\#\M_1= q^2+q+1.$$

Now we can replace another basic $J$-invariant and apply our estimation again. One can see that even the inequality is sharp for some basic $J$-invariant, the equality does not hold for all basic $J$-invariant.  Let us choose the basic $J$-invariant $J^{(q^2+q+1,0)}$, then one have

$$\frac{2(q^2+q+1)}{3}={\rm{ord}}_\fp(J_{\mathcal{O}_K}^{(0,q^2+q+1)})>\frac{(\delta_1+\delta_2)(q-1)}{1\cdot (q^3-1)\cdot 3}\sum_{n\geqslant 1} \#\M_n=\frac{q^2+q+1}{3}.$$

\subsection{The case  \texorpdfstring$\ \Delta=T(T+1)$, and $\pi=T$}
Again, one can check that $\mathcal{O}_K=A[\sqrt[3]{T^2+T}]$ by Theorem 1.1 in \cite{T01}.

Then the set $$\mathcal{B}:=\{(0,\beta,\gamma) \mid \beta, \gamma \in \F_{q^3}[T] \textrm{ with } {\rm Norm}(\beta)={\rm Norm}(\gamma)=1 \textrm{ and } {\rm Tr}(\beta\sigma(\gamma))=0\}$$ satisfy equation (7) to (9) when $n=1$. Furthermore, we have
$$\mathcal{B}\supset \{(0,\beta,1) \mid {\rm Norm}(\beta)=1, {\rm Tr}(\beta)=0\}\sqcup \{(0,1,\gamma) \mid {\rm Norm}(\gamma)=1, {\rm Tr}(\gamma)=0\}.$$

On the other hand, from Katz's estimation on Soto-Andrade sum (see Theorem 1.1 in \cite{MW10}), the number $N_3(0,1)$ of elements in $\F_{q^3}$ with norm equal to $1$ and trace equal to $0$ is bounded by the following inequality:
$$|N_3(0,1)-\frac{q^2-1}{q-1}|\leqslant g.c.d.(3,q-1)\sqrt{q}.$$

Hence we have 
$$N_3(0,1)\geqslant q+1-g.c.d(3,q-1)\sqrt{q},$$
which implies
$$
\begin{array}{lll}
\#\mathcal{B}&\geqslant& \#\{(0,\beta,1) \mid {\rm Norm}(\beta)=1, {\rm Tr}(\beta)=0\}+\# \{(0,1,\gamma) \mid {\rm Norm}(\gamma)=1, {\rm Tr}(\gamma)=0\} \\
\ &\geqslant &2N_3(0,1)\\
\ &\geqslant &2(q+1-g.c.d(3,q-1)\sqrt{q})
\end{array}.
$$

Therefore, we get

$$
\begin{array}{lll}
\nu_\fp(J_{\mathcal{O}_K}^{(\delta_1,\delta_2)})&\geqslant& \frac{(q+1)(q-1)}{3\cdot(q^3-1)\cdot 3}\sum_{m\geqslant 0} \#\mathcal{M}_{m\cdot e_{K,\fp}+1}\\
\ &\ &\ \\
\ &\geqslant& \frac{(q+1)(q-1)}{9\cdot(q^3-1)}\cdot 2(q+1-g.c.d(3,q-1)\sqrt{q})\\
\ &\ &\ \\
\ & =&\frac{2(q+1)(q+1-g.c.d(3,q-1)\sqrt{q})}{9\cdot(q^2+q+1)}

\end{array}.
$$

\section*{Acknowledgement}

The author would like to thank Professor Mihran Papikian and Professor Fu-Tsun Wei for helpful and inspiring discussions to carry out this paper.

\bibliographystyle{alpha}
\bibliography{singmodhrk_rev4.bib}

\end{document}